\documentclass[12pt]{amsart}

\usepackage{latexsym, amssymb, amsmath,ulem,soul,esint}
\usepackage{tikz}
\usepackage{hyperref}

\usepackage{amsfonts, graphicx}
\usepackage{graphicx,color}
\newcommand{\be}{\begin{equation}}
\newcommand{\ee}{\end{equation}}
\newcommand{\beq}{\begin{eqnarray}}
\newcommand{\eeq}{\end{eqnarray}}

\usepackage{wasysym,stmaryrd}
\theoremstyle{definition}

\newtheorem{thm}{Theorem}
\newtheorem{cl}{Corollary}

\newtheorem{prop}{Proposition}
\newtheorem{lm}{Lemma}
\newtheorem{clm}{Claim}

\newtheorem{rmk}{Remark}

\usepackage{mathrsfs}
\theoremstyle{remark}
\numberwithin{equation}{section}

\def\be{\begin{equation}}
\def\ee{\end{equation}}

\newcommand{\R}{\mathbb{R}}

\newcommand{\M}{\mathcal{M}^{n}}
\newcommand{\N}{\mathcal{N}}

\newcommand{\scal}{\mathcal{R}}
\newcommand{\n}{\text{\textbf{n}}}

\DeclareMathOperator{\Ric}{Ric}
\DeclareMathOperator{\Hess}{Hess}
\DeclareMathOperator{\vol}{Vol}
\DeclareMathOperator{\BigO}{O}

\DeclareMathOperator{\inj}{inj}
\DeclareMathOperator{\Rm}{Rm}

\DeclareMathOperator{\diam}{diam}


\begin{document}

\title[]
{Curvature estimates for steady and expanding solitons in higher dimensions}

\author{Pak-Yeung Chan}
\address[Pak-Yeung Chan]{
Department of Mathematics
National Tsing Hua University,
Hsinchu,
Taiwan
}
\email{pychan@math.nthu.edu.tw}

\author{Ming Hsiao}
\address[Ming Hsiao]{National Center for Theoretical Sciences,
Taipei 10617,
Taiwan
}
\email{minghsiao@ncts.ntu.edu.tw}

\renewcommand{\subjclassname}{
  \textup{2020} Mathematics Subject Classification}

\date{\today}

\begin{abstract}
In this paper, we demonstrate certain curvature estimates on complete non-compact steady and expanding gradient Ricci solitons in higher dimensions. In the expanding case, we prove that if the Ricci curvature decays at least quadratically, then the curvature operator decays at the rate $\BigO(1/r^{2})$ when $n=4$ and $\BigO((\log r)/r^{2})$ when $n\ge 5$. This refines the curvature bounds in a previous result by Cao-Liu-Xie, and removes the nonnegative Ricci curvature assumption in the estimates by Cao-Liu and Cao-Liu-Xie. As a geometric application, we establish the existence and uniqueness of $C^{1,\alpha}$ conical structure at infinity of Ricci expander with finite Ricci curvature ratio. 
In the steady case, using an integral estimate of the curvature, we prove that the curvature operator has at most polynomial growth when the potential function is proper and the Ricci curvature has linear decay. Moreover, we also confirm that the curvature is bounded if we further assume the Ricci curvature has super-linear decay $\BigO(r^{-1-\varepsilon})$. As an application, we prove the existence and uniqueness of cylindrical structure at infinity of steady soliton with super-linear Ricci curvature decay and proper potential function. 
\end{abstract}


\maketitle

\markboth{Pak-Yeung Chan, Ming Hsiao}{Curvature estimates for steady and expanding solitons in higher dimensions}

\section{Introduction}
The Ricci soliton is a natural generalization of an Einstein manifold and has been the subject of extensive study over the past four decades. As a self-similar solution to the Ricci flow, it plays a central role in the analysis of flow singularities. Consequently, understanding the geometry of Ricci soliton is of fundamental importance.

A gradient Ricci soliton $(\M,g,f)$ is a Riemannian manifold $(\M,g)$ with a smooth function $f\in\mathcal{C}^{\infty}(\M)$ such that 
\be
    \Ric_{g}+\Hess_{g}(f)=\frac{\lambda}{2}g,
\ee
for some constant $\lambda\in\R$. The soliton gradient Ricci soliton $(\M,g,f)$ is called \textit{shrinking, steady, expanding} Ricci soliton if \textit{$\lambda$ is positive, zero, negative}, respectively.
Upon a scaling, we can always normalize $\lambda$ such that \textit{$\lambda=1, 0, -1$} when the soliton is \textit{shrinking, steady, expanding} respectively. In this paper, we shall investigate the curvature of complete non-compact steady and expanding solitons in dimension no less than $4$.
\vskip0.2cm

A natural question concerning solitons is which curvature quantity plays the most essential role. For example, it is well known that for any complete shrinking or steady gradient Ricci soliton, the scalar curvature $\scal$ is always nonnegative \cite{C09}, whereas in the expanding case one has $\scal \geq -n/2$ \cite{PRS11, Z11}. Therefore, the soliton equation inherently encodes certain geometric constraints on the underlying manifold. On the other hand, if we require some decay of curvature, the structure at infinity of the soliton is relatively rigid. For example, Chen-Deruelle proved that for the expanding soliton, if the curvature has quadratic decay, then the expander must be asymptotic to a metric cone in $\mathcal{C}^{1,\alpha}$-sense \cite{CD15}. In shrinking cases, Munteanu-Wang showed that the shrinker is smoothly asymptotically conical if the Ricci tensor $|{\Ric}|$ decays to $0$ at infinity \cite{MW17} (see also \cite{KW15}). Deng-Zhu also proved that a $\kappa$-noncollapsed steady soliton with linear curvature decay and $\Rm\ge 0$ must be rotationally symmetric, and thus isometric to the Bryant soliton \cite{DZ20} (see also Brendle's result \cite{B13, B14}).
\vskip0.2cm

In this paper, we study the curvature estimate on complete non-compact gradient Ricci solitons. The weak curvature pinching estimates for gradient Ricci solitons have attracted considerable attention in the literature. See, for instance, \cite{CC20, CL22, C22, C19, CFSZ20, MW15, CX24, CZ22, Z21} for some recent progress in this direction. In dimension three, the Hamilton-Ivey estimates say any complete ancient Ricci flow of dimension three has nonnegative sectional curvature (see also \cite{CX25, CMZ25} for some pinching estimates on Ricci expanders). As a result, three dimensional complete Ricci shrinker and steadies must be nonnegatively curved. In dimension four, Munteanu-Wang \cite{MW15} used a dimension reduction technique via the level set of $f$ to show that a Ricci shrinker with bounded scalar curvature must satisfy $|{\Rm}|\leq C\scal$ for some positive constant $C$ (see also \cite{CC20}). In the absence of an analogous method in higher dimensions, Munteanu and Wang employed integration by parts and iteration techniques to investigate the geometry at infinity of Ricci shrinker (see also \cite{MW10}). In the expanding case, their approach was further developed by Cao-Liu \cite{CL22} and Cao-Liu-Xie \cite{CLX23}. 
Cao-Liu \cite{CL22} proved that for a gradient Ricci expander of dimension $4$ with nonnegative Ricci curvature and finite asymptotic scalar curvature ratio i.e. $\limsup_{x\to\infty}r^2\scal <\infty$, the asymptotic curvature ratio must be finite,
\[
\limsup_{x\to\infty}r^2 |{\Rm}|< \infty.
\]
Under the same curvature condition, Cao-Liu-Xie \cite{CLX23} generalized the result to dimension $n\ge 5$ and showed that the curvature tensor decays almost quadratically, more precisely for any $\alpha\in (0,2)$,
\[
\limsup_{x\to\infty}r^\alpha|{\Rm}|< \infty.
\]
\vskip0.2cm
Our first result removes the nonnegative Ricci curvature assumption in \cite{CL22, CLX23} for $n\ge 4$ and improves curvature decay rate to $\BigO(\frac{\log r}{r^{2}})$ for $n\ge 5$.

\begin{thm}\label{thm-expander with quadratic decay}
    Let $(\M,g,f)$ be an $n$-dimensional complete gradient expanding Ricci soliton with finite asymptotic Ricci curvature ratio
    \be\label{eq-quadratic decay of Ric}
    \limsup_{r\rightarrow\infty}|{\Ric}|r^{2}=:\eta<\infty,
    \ee
    where $n\ge 4$ and $r:=r(x)$ is the distance function to some fixed point $x_{0}\in\M$. Then for any $n\ge 5$, there is a constant $C>0$ such that
    \be
    |{\Rm}|(x)\leq\frac{C\log r(x)}{r(x)^{2}}
    \ee
    when $r(x)\gg 1$ is sufficiently large. Moreover, when $n=4$, the asymptotic curvature ratio is finite, i.e. for $r(x)\gg 1$,
    \be
    |{\Rm}|(x)\leq\frac{C}{r(x)^{2}}.
    \ee
\end{thm}
\vskip-0.1cm
As pointed out by Cao-Liu-Xie in \cite{CLX23}, it would be interesting to determine the sharp decay rate of the curvature tensor in this scenario. Motivated by the decay rate in dimension $4$, we expect that the optimal decay rate of $|{\Rm}|$ should also be quadratic $r^{-2}$ when $n\ge 5$. 

The proof of Theorem \ref{thm-expander with quadratic decay} consists of two key ingredients. The first one is to control the unit volume ratio $\vol(x,1)$ without assuming nonnegative Ricci curvature. Note that the lower bound of the volume ratio is crucial in applying Moser's iteration. In the nonnegative Ricci case, $\vol(x,1)$ has a uniform lower bound since the Ricci curvature is bounded \cite{CN09}. Without the nonnegative Ricci curvature condition, one needs to control $\vol(x,1)$ in an alternative way. To achieve this, we prove that the asymptotic $\mathcal{C}^{0}$-metric exists when the asymptotic Ricci curvature ratio is finite. Then the $\mathcal{C}^{0}$-convergence to a continuous metric provides a uniform lower bound for the volume of unit balls. This would suffice for the sub-linear Riemann curvature decay of curvature. The next ingredient is to analyze the evolution of curvature along the level set flow of $f$. This gives the curvature decay estimate of the form $O(\log r)/r^{2}$. In dimension 4, this estimate, combined with Munteanu–Wang's dimension-reduction method \cite{MW15}, leads to the quadratic decay of the curvature tensor.
\vskip0.2cm
With the boundedness of curvature tensor from Theorem \ref{thm-expander with quadratic decay}, we generalize a result in \cite[Theorem 2.10]{CD15} by removing the pinching conditions on $\Ric$ and $\nabla\Ric$. To present our result, we denote the functions $F:=-f$,  $z:=2\sqrt{F}$, the level set $\Sigma_{R}:=\{z=R\}$, and a family of diffeomorphisms $\varphi_{t-s}:\Sigma_{s}\rightarrow\Sigma_{t}$ with $\frac{d}{d\tau}\varphi_{\tau}=\frac{\nabla z}{|{\nabla z}|^{2}}$ with $\varphi_0=$ id for $t\geq s\gg1$ so that $\nabla z\neq0$.
\begin{cl}\label{cl-topology of ends on expander}
     Let $(\M,g,f)$ be an $n$-dimensional complete gradient expanding Ricci soliton with finite asymptotic Ricci curvature ratio (\ref{eq-quadratic decay of Ric}). Then outside a compact set $K\Subset\M$, $\M\setminus K$ is a disjoint union of a finite number of ends and is diffeomorphic to $\mathcal{N}^{n-1}\times(0,\infty)$, where $\mathcal{N}^{n-1}$ admits a metric $g_{\mathcal{N}}$ which is an almost $2\eta$-Einstein metric with positive normalized Ricci curvature, that is, on $\mathcal{N}^{n-1}$,
    \be
    \|{\Ric(g_{\mathcal{N}})-(n-2)g_{\mathcal{N}}}\|_{g_{\mathcal{N}}}<2\eta.
    \ee    
    More precisely, $(\mathcal{N}^{n-1},g_{\mathcal{N}}):=(\Sigma_{R},R^{-2}g|_{\Sigma_{R}})$ for any $R\geq R_{0}$ for some $R_{0}$.
\end{cl}
In fact, from Corollary \ref{cl-topology of ends on expander}, we show that the asymptotic tangent cone at infinity is unique and has a $\mathcal{C}^{1,\alpha}$-link $(\Sigma_{R_{0}},g_{\infty})$, which generalizes \cite[Theorem 1.2]{CD15}.
\begin{cl}\label{cl-structure of limit of fiber}
   Let $(\M,g,f)$ be an $n$-dimensional complete gradient expanding Ricci soliton with finite asymptotic Ricci ratio (\ref{eq-quadratic decay of Ric}). Then there is a $R_{0}>0$ and a $\mathcal{C}^{1,\alpha}$-Riemannian metric $g_{\infty}$ on $\Sigma_{R_{0}}$ such that 
   \be
        \varphi_{t-R_{0}}^{*}t^{-2}g|_{\Sigma_{t}}\overset{\mathcal{C}^{1,\alpha}}{\rightarrow}g_{\infty},
   \ee
   on $\Sigma_{R_{0}}$ as $t\rightarrow\infty$. Furthermore, for any $p\in\M$, 
   \be
        (\M,t^{-1}d_{g},p)\overset{\text{pGH}}{\longrightarrow}(C(\Sigma_{R_{0}}),dt^{2}+t^{2}g_{\infty}, o),
   \ee
   as $t\rightarrow\infty$, and the convergence is in $\ \mathcal{C}^{1,\alpha}_{\text{loc}}\left(C(\Sigma_{R_{0}})\setminus\{o\}\right)$-sense as $t\rightarrow\infty$ for any $\alpha\in(0,1)$.
\end{cl}
\begin{rmk}
    When $\eta=0$ in (\ref{eq-quadratic decay of Ric}), the smoothly conical structure at infinity of expander was also studied by Deruelle \cite{D17} (see also \cite{C12}).
\end{rmk}
\begin{rmk}\label{rmk-two notions of asymptotic cone} Without additional assumptions on the Ricci curvature of $g(t)$, it is currently unclear whether the weak asymptotic cone in the sense of \cite[p.5]{CL25} exists for the canonical flow generated by the expander. Despite that, if we choose $p\in\M$ such that $\nabla f(p)=0$, then the pointed Riemannian manifold $(\M,g(t),p)$ is isometric to $(\M,t^{-1}g,p)$ under the one-parameter family of diffeomorphisms $\frac{d\phi_{t}}{dt}=\frac{\nabla^{g}f}{t}$ with $\phi_{1}=\text{id}$. Therefore, as a consequence of Corollary \ref{cl-structure of limit of fiber}, 
    \be
    (\M,d_{g(t)},p)\overset{\text{pGH}}{\longrightarrow}(C(\Sigma_{R_{0}},g_{\infty}),o),
    \ee
    as $t\rightarrow0$ (see also \cite[Remark 1.5]{CD15}). However, without further assumption, it is not to be expected that there exists a metric $d$ on $\M$ such that $(\M,d)\simeq(C(X),d_{c})$ isometrically, or even at the topological level. 
\end{rmk}
\begin{rmk}
    In general, the number of ends could be greater than one for expanding solitons. Despite that, if $\scal\geq-\frac{n-1}{2}$ on $\M$, then by \cite{MW12}, either $\Sigma_{R_{0}}$ is connected or $\M\simeq\R\times\mathcal{N}^{n-1}$ for some compact Einstein manifold $\N^{n-1}$.
\end{rmk}

Very recently, Wu \cite{Wu25} has generalized the Munteanu-Wang result \cite{MW10} to the case of steady soliton. More precisely, Wu showed that for a steady gradient Ricci soliton with positive Ricci curvature and scalar curvature achieving maximum somewhere, the curvature tensor $\Rm$ grows at most exponentially in $r$, i.e. $|{\Rm}|(x)=O(e^{ar(x)})$ for some positive constant $a$. Motivated by this, as well as the corresponding estimate in shrinking and expanding cases, we are interested in studying the curvature estimates under a natural Ricci curvature decay condition, namely the linear Ricci curvature decay. Assuming the properness of the potential function $f$ and the linear Ricci curvature decay, we prove the following by a similar method as in Theorem \ref{thm-expander with quadratic decay}:
\begin{thm}\label{thm-steady with linear decay}
    Let $(\M,g,f)$ be an $n$-dimensional complete non-Ricci flat gradient steady Ricci soliton with $n\ge 5$. Suppose that $f$ is proper and the Ricci curvature decays linearly, i.e. $|{\Ric}|(x)\leq C_0/(r(x)+1)$ for some constant $C_0$ on $M$. Then the curvature tensor grows at most polynomially. That is, there exist two constants $C,N>0$ such that
    \be
    |{\Rm}|(x)\leq C(r(x)+1)^{N},
    \ee
    for all $x\in\M$, where $r(x)$ is the distance function to some fixes point $x_{0}\in\M$.
\end{thm}
\vskip-0.1cm
The properness condition $f\sim -r$ plays an essential role in our analysis, but it is unclear to us whether the result would still hold without it. On the one hand, the linear decay of the Ricci curvature is a generic condition that appears in many examples of steady solitons, such as the Bryant soliton and Appleton’s soliton \cite{A17, B05}, to name but a few.
\vskip0.2cm
It is natural to ask what the optimal bound for $|{\Rm}|$ should be if one assumes only conditions on the Ricci curvature and the potential function $f$ (see also \cite{C20} for estimates under the injectivity radius condition). On the product of Hamilton’s cigar soliton with any compact, non-flat but Ricci-flat manifold, the Ricci curvature decays exponentially and $f$ is proper, yet $|{\Rm}|$ remains bounded and does not decay at infinity. Hence, the best one can expect is the boundedness of $|{\Rm}|$. Indeed, by assuming a slightly faster decay of the Ricci curvature, we are able to confirm this.
\begin{cl}\label{cl-of-bdd-curvature}
    Let $(\M,g,f)$ be an $n$-dimensional complete non-Ricci flat gradient steady Ricci soliton. Suppose that $f$ is proper and there is an $\varepsilon>0$ such that 
    \be\label{Ricfast}
    |{\Ric}|(x)=O\left(r(x)^{-(1+\varepsilon)}\right),
    \ee    
    as $x\rightarrow\infty$, where $r(x)$ is the distance function to some fixed point $x_{0}\in\M$. Then the curvature tensor $|{\Rm}|$ is bounded. If in addition $\Ric\ge 0$ (or more generally $|{\Ric}|\leq C_0 \scal$) outside a compact set, then $|{\Ric}|=O(e^{-r(x)})$.
\end{cl}
\vskip 0.2cm
With the bounded curvature from Corollary \ref{cl-of-bdd-curvature}, we can discuss the asymptotic behavior at infinity. Denote $F:=-f$, the level set $\Sigma_{R}=\{F=R\}$, and the flow $\frac{d}{dt}\varphi_{t}=\frac{\nabla F}{|{\nabla F}|^{2}}$ on the region $|{\nabla F}|>0$. 
\begin{cl}\label{cl-asymptotic limit of super-linear decay}
    Under the same assumption as in Corollary \ref{cl-of-bdd-curvature}, there are a constant $R_{0}>0$ and a smooth and Ricci flat metric $g_{\infty}$ on $\Sigma_{R_{0}}$ such that
    \be
        \varphi_{R-R_{0}}^{*}g|_{\Sigma_{R}}\overset{\mathcal{C}^{1,\alpha}}{\longrightarrow}g_{\infty},
    \ee
    and
    \be
    \|{\Ric(g|_{\Sigma_{R}})}\|_{g|_{\Sigma_{R}}}\rightarrow0
    \ee
    as $R\rightarrow\infty$ for any $\alpha\in(0,1)$. Moreover, for any $p_{k}\rightarrow\infty$, the pointed manifolds $(\M,g,p_{k})$ converge to $(\R\times \Sigma_{R_{0}},dz^{2}+g_{\infty},(0,o))$ in the $\mathcal{C}^{\infty}_{\text{loc}}$-sense if $\varphi_{R_{0}-F(p_{k})}(p_{k})\rightarrow o$. Furthermore, if $\lim_{r\rightarrow\infty}|{\Rm}|=0$, then $(\Sigma_{R_{0}},g_{\infty})$ is a closed flat manifold. 
\end{cl}
\begin{rmk}
According to \cite{MW11}, a complete gradient steady soliton is either connected at infinity or isometric to $\R\times\N$ for some compact Ricci flat manifold $\N$. Therefore, $\Sigma_{R_{0}}$ is a closed and connected manifold.
\end{rmk}
The structure of this paper is as follows. In Section \ref{sec-preliminary}, we recall some basic properties of steady and expanding solitons. In Section \ref{sec-integral est}, we derive the main integral estimate of curvature in the steady soliton case. Finally, we prove our main results and corollaries in Section \ref{sec-proof of expander case} and Section \ref{sec-proof of steady case}.

\vskip0.2cm
{\it Acknowledgement}: The authors would like to thank Huai-Dong Cao, Chih-Wei Chen and Alix Deruelle for their interest in this work, and thank Man-Chun Lee for fruitful discussions. P.-Y. Chan is supported by the Yushan Young Fellow Program of the Ministry of Education (MOE) Taiwan (MOE-108-YSFMS-0004-012-P1), and by the NSTC grant 113-2115-M-007 -014 -MY2.

\section{Preliminaries}\label{sec-preliminary}
In this section, we recall some basic properties of steady and expanding solitons that we shall use in later sections.
\subsection{Steady solitons}
\begin{lm}\label{seqn}\cite{H95}
    Let $(\M,g,f)$ be a complete gradient steady Ricci soliton. Then
    \begin{align}
        &\scal+\Delta f=0;\\
        &\nabla_{i}\scal=2\Ric_{ij}\nabla_{j}f;\\
        &\label{normalized cond}\scal+|\nabla f|^{2}=C_{0},
    \end{align}
    for some constant $C_{0}$.
\end{lm}
$(\M,g,f)$ is said to be normalized gradient steady soliton if $C_{0}=1$. 
\begin{rmk}
    As a result of Chen \cite{C09}, the scalar curvature $\scal\geq0$ for a complete steady soliton. Therefore, $C_{0}=0$ only if $f$ is constant, which implies $(\M,g)$ is Ricci-flat. Hence, for a non-Ricci flat complete steady soliton, upon scaling the metric, we may as well assume it is normalized such that $C_0=1$.  
\end{rmk}
The next proposition elaborates under the assumption in Theorem \ref{thm-steady with linear decay}, $F:=-f$ is almost linear up to some small error.
\begin{prop}\cite[Theorem 3.2, Corollary 3.4]{CMZ22}\label{F is approximately r}
    Let $(\M,g,f)$ be a complete normalized non-Ricci-flat gradient steady Ricci soliton with $n\geq4$. Suppose that $f$ is proper and the scalar curvature decays linearly, i.e. $\scal\leq\frac{C_{1}}{r+1}$ for some positive constant $C_1$. Then there is a constant $C>0$ such that
    \be\label{steadyf}
    r-C_{2}\log (r+1)-C_{2}\leq F:=-f\leq r+C_{2},
    \ee
    on $\M$, where $r(\cdot):=d_{g}(\cdot,p)$ is the distance function based at $p\in\M$. 
\end{prop}

To proceed, we recall some basic evolution equations of various curvature quantities.
\begin{lm}\cite[Lemma 2.4]{CC20}
Let $(\M,g,f)$ be an $n$-dimensional complete gradient steady Ricci soliton. Then
\begin{align}
    &\label{nable-Ric}\nabla_{j}\Rm_{ilkj}=\nabla_{i}\Ric_{kl}-\nabla_{l}\Ric_{ki}=\Rm_{ilkj}f_{j}\\
    &\label{Laplacian-Ric}\Delta|{\Ric}|^{2}\geq2|{\nabla\Ric}|^{2}+\langle\nabla f,\nabla|{\Ric}|^{2}\rangle-4|{\Rm}||{\Ric}|^{2}\\
    &\label{Laplacian-Rm}\Delta|{\Rm}|^{2}\geq2|{\nabla\Rm}|^{2}+\langle\nabla f,\nabla|{\Rm}|^{2}\rangle-C_{3}|{\Rm}|^{3},
\end{align}
here $C_3$ is a positive dimensional constant.
\end{lm}
Finally, we recall that under the same assumption as Theorem \ref{thm-steady with linear decay}, the volume of geodesic balls grow at most polynomially.
\begin{rmk}\label{rms}
    Under the same assumption of Theorem \ref{thm-steady with linear decay}, Chan-Ma-Zhang \cite{CMZ22} showed that the steady soliton is of polynomial volume growth. Hence by choosing larger $N$ and $C$ if necessary, we have for all large $r$
    \[
    \vol(p,r)\leq Cr^N.
    \]
\end{rmk}
\subsection{Expanding soliton}
Throughout this subsection, we denote $F:=-f$.
\begin{lm}\label{seqn-expander}\cite{H95,C09}
    Let $(\M,g,f)$ be a complete gradient expanding Ricci soliton. Then
    \begin{align}
        &\scal=\Delta F-\frac{n}{2};\\
        &\nabla_{i}\scal=-2\Ric_{ij}\nabla_{j}F;\\
        &\scal+|\nabla F|^{2}=F+C_{0};\\
        &\scal\geq-\frac{n}{2}.
    \end{align}
\end{lm}
By adding a constant to $f$ if necessary, we may assume, without loss of generality, that $C_{0}=-\frac{n}{2}$.

Under the assumption in Theorem \ref{thm-expander with quadratic decay}, the function $F$ is almost the quarter of distance square and has finite asymptotic volume ratio. 
\begin{prop}\cite[Section 2]{CD15}\label{F is approximately r^2/4}
    Let $(\M,g,f)$ be a complete gradient expanding Ricci soliton with finite asymptotic Ricci curvature (\ref{eq-quadratic decay of Ric}). Then 
    \be\label{eq-expanding F like d2/4}
      \frac{r(x)^{2}}{4}-\frac{\pi\eta}{2}r(x)+F(p)-|{\nabla F(p)}|\leq F(x)\leq \left(\frac{r(x)}{2}+\sqrt{F(p)}\right)^{2},
    \ee
    on $\M$, where $r(\cdot):=d_{g}(\cdot,p)$ is the distance function based at $p\in\M$. Moreover, for any $x\in\M$
    \be
    \lim_{r\rightarrow\infty}\frac{\vol(x,r)}{r^{n}}<+\infty.
    \ee
\end{prop}
We conclude this section by recalling some basic evolution equations for various curvature quantities on a Ricci expander.
\begin{lm}\cite[Lemma 2.4]{CL22}
Let $(\M,g,f)$ be an $n$-dimensional complete gradient expanding Ricci soliton. Then
\begin{align}
    &\label{nable-Ric-expander}\nabla_{j}\Rm_{ilkj}=\nabla_{i}\Ric_{kl}-\nabla_{l}\Ric_{ki}=\Rm_{ilkj}f_{j}\\
    &\label{Laplacian-Ric-expander}\Delta|{\Ric}|^{2}\geq2|{\nabla\Ric}|^{2}+\langle\nabla f,\nabla|{\Ric}|^{2}\rangle-2|{\Ric}|^{2}-4|{\Rm}||{\Ric}|^{2}\\
    &\label{Laplacian-Rm-expander}\Delta|{\Rm}|^{2}\geq2|{\nabla\Rm}|^{2}+\langle\nabla f,\nabla|{\Rm}|^{2}\rangle-2|{\Rm}|^{2}-C_{4}|{\Rm}|^{3},
\end{align}
\end{lm}

\section{The integral estimate on steady soliton}\label{sec-integral est}
In this section, we assume $(\M,g,f)$ is a complete normalized gradient steady Ricci soliton and the following:
\begin{itemize}
    \item $f$ is proper and negative (after subtracting a sufficiently large constant from $f$, see \eqref{steadyf});
    \item $|{\Ric}|\leq R_{0}/(r(x)+1)$, where $r(x):=d_{g}(x,p)$ is the distance function based at $p$ and $R_0$ is a nonnegative constant. 
\end{itemize}
Motivated by \cite{MW10,CLX23, Wu25}, for $r\gg1$, we define
\be
D(r):=\{x\in\M:F(x)\leq r\}, 
\ee
and the cut-off function
\be
\phi(x):=\begin{cases}
    \frac{r-F(x)}{r},\text{ if }x\in D(r);\\
    0,\text{ if }x\notin D(r).
\end{cases}
\ee
Note that according to Proposition \ref{F is approximately r}, $D(r)$ is a compact set. Now by our hypothesis and \eqref{normalized cond}, there exist constants $\delta>0$ and $r_{0}>0$  such that 
\be\label{cond-of-r0}
    |{\Ric}|\leq \frac{\delta}{F}\text{ and }|\nabla F|^{2}\geq1/2\text{ outside }D(r_{0}).
\ee
In general, we could always choose $\delta\leq2R_{0}$. Let $p>1$ and $q:=2p+1$.
From now on, we adopt the following convention:
\begin{itemize}
    \item $C$: a constant that depends on $D(r_{0}),n,R_{0},p,\delta$ and $a$.
    \item $c$: a constant that depends on $n,R_{0}$.
\end{itemize}
Furthermore, unless otherwise specified, the values of these constants may vary from line to line.
\begin{prop}\label{int nabla Ric}
    Under the same assumptions and notations as above, for any $|a|>10$, if $p>|a|+N+1$ where $N$ is the polynomial degree of volume growth as in Remark \ref{rms}, then  
    \be
    \begin{aligned}
        \int_{\M}|{\nabla\Ric}|^{2}|{\Rm}|^{p-1}F^{a+2}\phi^{q}\leq&~ c\delta^{2}p^{2}\int_{\M}|{\nabla\Rm}|^{2}|{\Rm}|^{p-3}F^{a}\phi^{q}+c\delta^{2}\int_{\M}|{\Rm}|^{p}F^{a}\phi^{q}\\
        &~+c\delta^{2}\int_{\M}|{\Rm}|^{p-1}F^{a}\phi^{q}+C.
    \end{aligned}
    \ee
\end{prop}
\begin{proof}[Proof of Proposition \ref{int nabla Ric}]
    By (\ref{Laplacian-Ric}),
    \begin{equation}\label{eq-1-Prop2}
        \begin{aligned}
            2\int_{M}|{\nabla\Ric}|^{2}|{\Rm}|^{p-1}F^{a+2}\phi^{q}\leq&~\int_{\M}(\Delta|{\Ric}|^{2})|{\Rm}|^{p-1}F^{a+2}\phi^{q}\\
            &~+2\int_{\M}\langle\nabla F,\nabla|{\Ric}|^{2}\rangle|{\Rm}|^{p-1}F^{a+2}\phi^{q}\\
            &~+4\int_{\M}|{\Ric}|^{2}|{\Rm}|^{p}F^{2+a}\phi^{q}\\
            =:&~I+II+III.
        \end{aligned}
    \end{equation}
    For the last term, note that $|{\Ric}|\leq \delta/F$ on $D(r_{0})^{c}$, this implies
    \be\label{eq-2-Prop2}
    III\leq c\delta^{2}\int_{\M}|{\Rm}|^{p}F^{a}\phi^{q}+C.
    \ee
    For the second term, using the fact $|\nabla F|\leq1$, $|{\Ric}|\leq\delta/F$ outside $D(r_{0})$, and the Cauchy inequality, we conclude that
    \be\label{eq-3-Prop2}
    \begin{aligned}
    II\leq&~\frac{1}{4}\int_{M}|{\nabla\Ric}|^{2}|{\Rm}|^{p-1}F^{a+2}\phi^{q}+c\delta^{2}\int_{\M}|{\Rm}|^{p-1}F^{a}\phi^{q}+C.
    \end{aligned}
    \ee
    Lastly, we apply integration by parts and Kato's inequality to the first term, 
    \be
    \begin{aligned}\label{eq-4-Prop2}
        I=&~-\int_{\M}\langle\nabla|{\Ric}|^{2},\nabla|{\Rm}|^{p-1}\rangle F^{a+2}\phi^{q}\\
        &~-\int_{\M}\langle\nabla|{\Ric}|^{2},\nabla F^{a+2}\rangle |{\Rm}|^{p-1}\phi^{q}\\
        &~-\int_{\M}\langle\nabla|{\Ric}|^{2},\nabla\phi^{q}\rangle|{\Rm}|^{p-1} F^{a+2}\\
        \leq&~2(p-1)\int_{\M}|{\nabla\Ric}||{\nabla\Rm}||{\Ric}||{\Rm}|^{p-2}F^{a+2}\phi^{q}\\
        &~+2|a+2|\int_{\M}|{\nabla\Ric}||{\nabla F}||{\Ric}||{\Rm}|^{p-1}F^{a+1}\phi^{q}\\
        &~+\frac{(4p+2)}{r}\int_{\M}|{\nabla\Ric}||{\nabla F}||{\Ric}||{\Rm}|^{p-1}F^{a+2}\phi^{q-1}\\
        =:&~\mathfrak{A}+\mathfrak{B}+\mathfrak{C}.
    \end{aligned}
    \ee
    For the term $\mathfrak{A}$, similar to above, by Ricci bounds and the Cauchy inequality,
    \be\label{eq-5-Prop2}
    \mathfrak{A}\leq\frac{1}{4}\int_{\M}|{\nabla\Ric}|^{2}|{\Rm}|^{p-1}F^{a+2}\phi^{q}+c\delta^{2}p^{2}\int_{\M}|{\nabla\Rm}|^{2}|{\Rm}|^{p-3}F^{a}\phi^{q}+C.
    \ee
    For $\mathfrak{B}$, by $|\nabla F|\leq1$ and the Ricci bounds,  
    \be
    \begin{aligned}
        \mathfrak{B}\leq&~\frac{1}{4}\int_{\M}|{\nabla\Ric}|^{2}|{\Rm}|^{p-1}F^{a+2}\phi^{q}+c\delta^{2}a^{2}\int_{\M}|{\Rm}|^{p-1}F^{a-2}\phi^{q}+C\\
        \leq&~\frac{1}{4}\int_{\M}|{\nabla\Ric}|^{2}|{\Rm}|^{p-1}F^{a+2}\phi^{q}+\delta^{2}\int_{\M}|{\Rm}|^{p}F^{a}\phi^{q}+c\delta^{2}|a|^{2p}\int_{\M}F^{a-2p}\phi^{q}+C,
    \end{aligned}
    \ee
    where the last inequality follows from Young's inequality. By the choice of $p$ and the at most polynomial volume growth, $|a|^{2p}\int_{\M}F^{a-2p}\phi^{q}$ is bounded by a constant independent of $r$. Consequently, 
    \be\label{eq-6-Prop2}
    \mathfrak{B}\leq\frac{1}{4}\int_{\M}|{\nabla\Ric}|^{2}|{\Rm}|^{p-1}F^{a+2}\phi^{q}+\delta^{2}\int_{\M}|{\Rm}|^{p}F^{a}\phi^{q}+C.
    \ee
    Lastly, since $F\leq r+C\leq Cr$ on $D(r)$, and by combining with $|\nabla F|\leq1$ and the Ricci bounds, we see that
    \be\label{eq-7-Prop2}
    \begin{aligned}
        \mathfrak{C}\leq&~C\delta p\int_{\M}|{\nabla\Ric}||{\Rm}|^{p-1}F^{a}\phi^{q-1}+C\\
        \leq&~\frac{1}{4}\int_{\M}|{\nabla\Ric}|^{2}|{\Rm}|^{p-1}F^{a+2}\phi^{q}+C\delta^{2}p^2\int_{\M}|{\Rm}|^{p-1}F^{a-2}\phi^{q-2}+C\\
        \leq&~\frac{1}{4}\int_{\M}|{\nabla\Ric}|^{2}|{\Rm}|^{p-1}F^{a+2}\phi^{q}+\delta^{2}\int_{\M}|{\Rm}|^{p}F^{a}\phi^{q}+C\delta^{2}p^{2p}\int_{\M}F^{a-2p}\phi^{q-2p}+C,
    \end{aligned}
    \ee
    we have just applied the Cauchy's inequality and Young's inequality in the last two lines. As above, $\int_{\M}F^{a-2p}\phi^{q-2p}$ is bounded from above due to the choice of $p$ and the at most polynomial volume growth. The result then follows by combining all the above estimates. 
\end{proof}
\begin{prop}\label{int nabla Rm}
Under the same conditions as Proposition in \ref{int nabla Ric}, and assuming $p\geq3$ further,
\be
\int_{\M}|{\nabla\Rm}|^{2}|{\Rm}|^{p-3}F^{a}\phi^{q}\leq c\int_{\M}|{\Rm}|^{p-1}F^{a}\phi^{q}+c\int_{\M}|{\Rm}|^{p}F^{a}\phi^{q}+C.
\ee
\end{prop}
\begin{proof}[Proof of Proposition \ref{int nabla Rm}]
    By (\ref{Laplacian-Rm}), we have
    \be
    \begin{aligned}
    2\int_{\M}|{\nabla\Rm}|^{2}|{\Rm}|^{p-3}F^{a}\phi^{q}\leq&~\int_{\M}(\Delta|{\Rm}|^{2})|{\Rm}|^{p-3}F^{a}\phi^{q}\\
    &~+\int_{\M}\langle\nabla F,\nabla|{\Rm}|^{2}\rangle|{\Rm}|^{p-3}F^{a}\phi^{q}\\
    &~+c\int_{\M}|{\Rm}|^{p}F^{a}\phi^{q}\\
    =:&~I+II+III.
    \end{aligned}
    \ee
    For the second term, we adopt the Cauchy inequality and $|{\nabla F}|\leq1$, 
    \be
    II\leq \frac{1}{3}\int_{\M}|{\nabla\Rm}|^{2}|{\Rm}|^{p-3}F^{a}\phi^{q}+3\int_{\M}|{\Rm}|^{p-1}F^{a}\phi^{q}.
    \ee
    After the integration by parts, $I$ becomes
    \be
    \begin{aligned}
        I=&~-(p-3)\int_{\M}\langle\nabla|{\Rm}|^{2},\nabla|{\Rm}|\rangle|{\Rm}|^{p-4}F^{a}\phi^{q}\\
        &~-a\int_{\M}\langle\nabla|{\Rm}|^{2},\nabla F\rangle|{\Rm}|^{p-3}F^{a-1}\phi^{q}\\
        &~+\frac{q}{r}\int_{\M}\langle\nabla|{\Rm}|^{2},\nabla F\rangle|{\Rm}|^{p-3}F^{a}\phi^{q-1}\\
        =:&~\mathfrak{A}+\mathfrak{B}+\mathfrak{C}.
    \end{aligned}
    \ee
    For $\mathfrak{A}$, notice that it is non-positive when $p\geq3$, we obtain $\mathfrak{A}\leq0$. For $\mathfrak{B}$, using the fact $|{\nabla F}|\leq1$ and Cauchy's inequality,
    \be
    \mathfrak{B}\leq \frac{1}{3}\int_{\M}|{\nabla\Rm}|^{2}|{\Rm}|^{p-3}F^{a}\phi^{q}+ca^{2}\int_{\M}|{\Rm}|^{p-1}F^{a-2}\phi^{q}.
    \ee
    Applying Young's inequality on the last term, 
    \be
    ca^{2}\int_{\M}|{\Rm}|^{p-1}F^{a-2}\phi^{q}\leq\int_{\M}|{\Rm}|^{p}F^{a}\phi^{q}+Ca^{2p}\int_{\M}F^{a-2p}\phi^{q}\leq\int_{\M}|{\Rm}|^{p}F^{a}\phi^{q}+C,
    \ee
    where the last term follows from the at most polynomial volume growth. For $\mathfrak{C}$, using the fact $F\leq r$ on $D(r)$ and $|\nabla F|\leq1$,
    \be
    \begin{aligned}
        \mathfrak{C}\leq&~ cp\int_{\M}|{\nabla\Rm}||{\Rm}|^{p-2}F^{a-1}\phi^{q-1}\\
        \leq&~\frac{1}{3}\int_{\M}|{\nabla\Rm}|^{2}|{\Rm}|^{p-3}F^{a}\phi^{q}+cp^{2}\int_{\M}|{\Rm}|^{p-1}F^{a-2}\phi^{q-2}\\
        \leq&~\frac{1}{3}\int_{\M}|{\nabla\Rm}|^{2}|{\Rm}|^{p-3}F^{a}\phi^{q}+\int_{\M}|{\Rm}|^{p}F^{a}\phi^{q}+Cp^{2p}\int_{\M}F^{a-2p}\phi^{q-2p},
    \end{aligned}
    \ee  
    we have just applied the Cauchy's inequality and Young's inequality in the last two inequalities, respectively. Since $p>|a|+N+1$ and $q=2p+1$, the last term is bounded from above by $C$ due to the polynomial volume growth of $\M$. Combining all the estimates above, we confirm the assertion.
\end{proof}
\vskip-0.1cm

With the above preparations in place, we are now ready to state the main integral estimate of this section.

\begin{thm}\label{int-estimate}
Let $(\M,g,f)$ be an $n$-dimensional complete normalized steady gradient Ricci soliton. Suppose $f$ is proper and the Ricci curvature has linear decay
\be\label{ricc}
|{\Ric}|(x)\leq\frac{R_{0}}{1+d_{g}(x,x_{0})},
\ee
for all $x\in\M$ and for some $x_{0}\in\M$, where $R_0$ is a positive a constant. Then for $p+a \geq N+1$, $a>10$ and $p\geq3$, there exist some constants $c(n,R_{0}),C_{1}(n,p,a,g)>0$ such that
\be
\left[\frac{1}{2}a-c(\delta p)^{3}-c\delta p^{5}-c\right]\int_{\M}|{\Rm}|^{p}F^{-a}d\vol_g\leq C_{1},
\ee
here $\delta$ is the positive constant in (\ref{cond-of-r0}) and $N$ is the polynomial degree of volume growth in Remark \ref{rms}.
\end{thm}
\begin{proof}[Proof of Theorem \ref{int-estimate}]
    Under the setup described above, we can take $r_{0}>0$ such that (\ref{cond-of-r0}) holds. Then by (\ref{ricc}) and Lemma \ref{seqn},
    \be
    \begin{aligned}
    -cR_{0}\int_{\M}|{\Rm}|^{p}F^{-a}\phi^{q}\leq&~-\int_{\M}|{\Rm}|^{p}\scal F^{1-a}\phi^{q}\\
    =&~-\int_{\M}|{\Rm}|^{p}(\Delta F)F^{1-a}\phi^{q}\\
    =&~\int_{\M}\langle\nabla|{\Rm}|^{p},\nabla F\rangle F^{1-a}\phi^{q}\\
    &~+(1-a)\int_{\M}|{\nabla F}|^{2}|{\Rm}|^{p}F^{-a}\phi^{q}\\
    &~-\frac{q}{r}\int_{\M}|{\nabla F}|^{2}|{\Rm}|^{p}F^{1-a}\phi^{q-1}
    \\
    =:&~I+II+III,
    \end{aligned}
    \ee
where we take the integration by parts at the second equality. Evidently, $III\leq0$. For $II$, by the choice of $r_{0}$, we have
\be
II\leq\frac{1-a}{2}\int_{\M}|{\Rm}|^{p}F^{-a}\phi^{q}+C.
\ee
For $I$, notice the identity
\be
\langle\nabla|{\Rm}|^{p},\nabla F\rangle=pR_{ijk\ell}R_{ijk\ell,m}F_{m}|{\Rm}|^{p-2}=2pR_{ijk\ell}R_{ijkm,\ell}F_{m}|{\Rm}|^{p-2},
\ee
which follows from the second Bianchi identity. Then we perform integration by parts on $I$,
\be
\begin{aligned}
    I=&~-2p\int_{\M}R_{ijkm}F_{m}R_{ijk\ell,\ell}|{\Rm}|^{p-2}F^{1-a}\phi^{q}\\
    &~-2p\int_{\M}R_{ijkm}R_{ijk\ell}F_{m\ell}|{\Rm}|^{p-2}F^{1-a}\phi^{q}\\
    &~-2p\int_{\M}R_{ijkm}F_{m}R_{ijk\ell}(\nabla_{\ell}|{\Rm}|^{p-2})F^{1-a}\phi^{q}\\
    &~-2p(1-a)\int_{\M}R_{ijkm}F_{m}R_{ijk\ell}F_{\ell}|{\Rm}|^{p-2}F^{-a}\phi^{q}\\
    &~+\frac{2pq}{r}\int_{\M}R_{ijkm}F_{m}R_{ijk\ell}F_{\ell}|{\Rm}|^{p-2}F^{1-a}\phi^{q-1}\\
    =:&~\mathfrak{A}+\mathfrak{B}+\mathfrak{C}+\mathfrak{D}+\mathfrak{E}.
\end{aligned}
\ee
For $\mathfrak{B}$, since $|{\Ric}|=|{\Hess F}|\leq \delta/F$ outside $D(r_{0})$,
\be
\mathfrak{B}\leq c\delta p\int_{\M}|{\Rm}|^{p}F^{-a}\phi^{q}+C.
\ee
By (\ref{nable-Ric}), $|\nabla F|\leq1$ and $F\leq r$ on $D(r)$,
\be
\begin{aligned}
    \mathfrak{A}+\mathfrak{D}+\mathfrak{E}\leq&~\int_{\M}|{\nabla\Ric}||{\Rm}|^{p-1}(cpF^{1-a}\phi^{q}+cpa F^{-a}\phi^{q}+cp^{2}F^{-a}\phi^{q-1})\\
    \leq&~\delta p\int_{\M}|{\nabla\Ric}|^{2}|{\Rm}|^{p-1}F^{2-a}\phi^{q}+c\delta^{-1}p\int_{\M}|{\Rm}|^{p-1}F^{-a}\phi^{q}\\
    &~+c\delta^{-1}pa^{2}\int_{\M}|{\Rm}|^{p-1}F^{-2-a}\phi^{q}+c\delta^{-1}p^{3}\int_{\M}|{\Rm}|^{p-1}F^{-2-a}\phi^{q-2}\\
    \leq&~\delta p\int_{\M}|{\nabla\Ric}|^{2}|{\Rm}|^{p-1}F^{2-a}\phi^{q}+c\delta^{-1}p^{3}\int_{\M}|{\Rm}|^{p-1}F^{-a}\phi^{q-2}+C,
\end{aligned}
\ee
where we apply Cauchy's inequality to the second line. For $\mathfrak{C}$, by applying (\ref{nable-Ric}) and Cauchy's inequality again, we have
\be
\begin{aligned}
\mathfrak{C}\leq&~cp^{2}\int_{\M}|{\nabla\Ric}||{\nabla\Rm}||{\Rm}|^{p-2}F^{1-a}\phi^{q}\\
    \leq&~\delta^{-1}p^{3}\int_{\M}|{\nabla\Ric}|^{2}|{\Rm}|^{p-1}F^{2-a}\phi^{q}+c\delta p\int_{\M}|{\nabla\Rm}|^{2}|{\Rm}|^{p-3}F^{-a}\phi^{q}.
\end{aligned}
\ee
In summary, we obtain
\be
\begin{aligned}
\left(\frac{1}{2}a-\delta p-cR_{0}-\frac{1}{2}\right)\int_{\M}|{\Rm}|^{p}F^{-a}\phi^{q}\leq&~(c\delta p+\delta^{-1}p^{3})\int_{\M}|{\nabla\Ric}|^{2}|{\Rm}|^{p-1}F^{2-a}\phi^{q}\\
&~+c\delta p\int_{\M}|{\nabla\Rm}|^{2}|{\Rm}|^{p-3}F^{-a}\phi^{q} \\
&~+c\delta^{-1}p^{3}\int_{\M}|{\Rm}|^{p-1}F^{-a}\phi^{q-2}+C.
\end{aligned}
\ee
Combining this with Proposition \ref{int nabla Ric} and Proposition \ref{int nabla Rm}, 
\be
\left(\frac{1}{2}a-c(\delta p)^{3}-c\delta p^{5}-c\right)\int_{\M}|{\Rm}|^{p}F^{-a}\phi^{q}\leq c(\delta,p)\int_{\M}|{\Rm}|^{p-1}F^{-a}\phi^{q-2}+C,
\ee
where $c(\delta,p)>0$ is a constant. By Young's inequality, 
\be
c(\delta,p)\int_{\M}|{\Rm}|^{p-1}F^{-a}\phi^{q-2}\leq c\delta p^{5}\int_{\M}|{\Rm}|^{p}F^{-a}\phi^{q}+c(\delta,p)\int_{\M}F^{-a-2p}\phi^{q-2p}.
\ee
Since the last term is bounded from above by the polynomial volume growth, this indicates
\be
\left(\frac{1}{2}a-c(\delta p)^{3}-c\delta p^{5}-c\right)\int_{\M}|{\Rm}|^{p}F^{-a}\phi^{q}\leq C.
\ee
Therefore, the result follows by letting $r\rightarrow\infty$.
\end{proof}

\section{Main Theorems - expander}\label{sec-proof of expander case}
Before we prove the theorem \ref{thm-expander with quadratic decay}, we recall an integral curvature estimate in \cite[Proposition 3.1]{CLX23}.
\begin{prop}\label{int-estimate-expander}\cite[Proposition 3.1 and Remark 3.1]{CLX23}
Let $(\M,g,f)$ be an $n$-dimensional complete expanding gradient Ricci soliton with finite asymptotic Ricci curvature ratio
\be
\limsup_{x\rightarrow\infty}|{\Ric}|r(x)^{2}<\infty.
\ee
Then for any $a>0$, there is a constant $c>0$ such that for $p>a+c$, 
\be
\left[1-p^{-1}(a+c)\right]\int_{\M}|{\Rm}|^{p}F^{a}d\vol_g\leq c(p),
\ee
where $c(p)$ is of order $p^p$ and $F:=\frac{n}{2}-f$.
\end{prop}
We point out that, as \cite[Remark 3.1]{CLX23} elaborates, although the original statement assumes $\Ric\geq0$ and finite asymptotic scalar curvature, the assertion indeed holds under the finite asymptotic Ricci curvature ratio without $\Ric\ge 0$. Since the quadratic decay of Ricci curvature ensures the polynomial volume growth and the potential function $F$ has quadratic growth by Proposition \ref{F is approximately r^2/4}, all the arguments in the proof of \cite[Proposition 3.1]{CLX23} remain valid.

The following proposition shows that the finite asymptotic Ricci curvature ratio implies the existence of asymptotic cone at infinity, which is the first part of Corollary \ref{cl-structure of limit of fiber}. In particular, this gives a uniform lower bound for unit ball volume on $\M$.
\begin{prop}\label{prop-vol lower bound of expander}
    Let $(\M,g,f)$ be an $n$-dimensional complete expanding gradient Ricci soliton with finite asymptotic Ricci curvature ratio
\be
\limsup_{x\rightarrow\infty}|{\Ric}|r(x)^{2}=:\eta<\infty.
\ee
Then there exist a constant $R_{0}>0$ and a $\mathcal{C}^{0}$-Riemannian metric $g_{\infty}$ on $\Sigma_{R_{0}}$ such that 
   \be
        \varphi_{t-R_{0}}^{*}t^{-2}g|_{\Sigma_{t}}\overset{\mathcal{C}^{0}}{\rightarrow}g_{\infty},
   \ee
   on $\Sigma_{R_{0}}$ as $t\rightarrow\infty$. Moreover, $(\M,t^{-2}g,p)$ converges to $(C(\Sigma_{R_{0}},g_{\infty}),o)$ in the Gromov-Hausdorff sense and $\mathcal{C}^{0}_{\text{loc}}(C(\Sigma_{R_{0}})\setminus\{o
   \})$-sense as $t\rightarrow\infty$. In particular, there exists a constant $v>0$ such that $\vol_g(x,1)\geq v$ for all $x\in\M$.
\end{prop}
\begin{proof}[Proof of Proposition \ref{prop-vol lower bound of expander}]
    We first prove the $\mathcal{C}^{0}$-limit of the level set exists and we follow the notation in the statement of Corollary \ref{cl-topology of ends on expander}. 
    Since 
    \begin{eqnarray*}
    \left|\frac{\nabla z}{|{\nabla z}|^{2}}\right|=\frac{\sqrt{F}}{|{\nabla F}|}=\left(1+\frac{\scal+\frac{n}{2}}{F-\scal-\frac{n}{2}}\right)^{1/2}&=&1+\frac{1}{2}\left(\frac{n}{2}\cdot\frac{4}{r^{2}}\right)+\BigO(r^{-4})\\
    &=&1+nr(x)^{-2}+\BigO(r(x)^{-4}),
    \end{eqnarray*}
    as $r(x)\rightarrow\infty$, this shows that $\Sigma_{R}$ is a smooth hypersurface for sufficiently large $R$ (say $R\geq R_{0}$) and $z=r(x)+\BigO(1)$ as $r(x)\rightarrow\infty$. Also, they are diffeomorphic to each other via the flow $\varphi_{R_{2}-R_{1}}:\Sigma_{R_{1}}\rightarrow \Sigma_{R_{2}}$ with vector field $\frac{\partial\varphi_{t}}{\partial t}=\frac{\nabla z}{|{\nabla z}|^{2}}$. Recall that we have the following results from \cite[Lemma 2.5, 2.6, 2.7 and Corollary 3.2]{CD15}.
    \begin{clm}\label{cl-structure on fiber C0}
        There is a constant $C>0$ (may depending on the geometry of $(\M,g,f)$) such that the following properties about $(\Sigma_{r},r^{-2}g|_{\Sigma_{r}})$ hold for all $r>R_{0}$.
        \begin{enumerate}
            \item\label{cl2-vol} (Volume Control) For $r>R_{0}$,
            \be
                C^{-1}\leq\vol(\Sigma_{r},r^{-2}g|_{\Sigma_{r}})\leq C.
            \ee
            \item\label{cl2-diam} (Diameter Control) For $r>R_{0}$,
            \be
            \diam(\Sigma_{r},r^{-2}g|_{\Sigma_{r}})\leq C,
            \ee
            where the diameter here means the sum of the intrinsic diameters of each connected component.
            \item\label{cl2-GH approx} (Gromov-Hausdorff approximation) For $t\geq s\geq R_{0}$, the map $\varphi_{t-s}:(\Sigma_{s},s^{-2}g|_{\Sigma_{s}})\rightarrow(\Sigma_{t},t^{-2}g|_{\Sigma_{t}})$ is an $o(1)$-Gromov-Hausdorff approximation as $s\rightarrow\infty$.
        \end{enumerate}
    \end{clm}
            Denote the second fundamental from on $\Sigma_{R}$ with respect to the normal vector $\n=\nabla z/|{\nabla z}|$ by $h_{R}$. Then $h_{R}=\Hess z/|{\nabla z}|$ since $\Sigma_{R}=z^{-1}(R)$. Note that we have  $F=\frac{r(x)^{2}}{4}+\BigO(r(x))$, $|{\nabla F}|=\frac{r(x)}{2}+\BigO(1)$, and $\Hess{F}-\frac{1}{2}g=(\eta+o(1))r(x)^{-2}$. These imply $R=z=r(x)+\BigO(\sqrt{r})$ and $|{\nabla z}|=\frac{|{\nabla F}|}{\sqrt{F}}=1-nR^{-2}+o(R^{-2})$ as $R\rightarrow\infty$. Consider the metric $g_{R}:=z^{-2}g|_{\Sigma_{R}}$ on the slice $\Sigma_{R}:=\{z=R\}$. Since 
    \be\label{eq-z}
    \frac{\partial}{\partial z}=\frac{1}{|{\nabla z}|}\n=\left(1+\frac{-\scal-\frac{n}{2}}{F}\right)^{-\frac{1}{2}}\n=(1+nz^{-2}+o(z^{-2}))\n
    \ee
    as $r\rightarrow\infty$. Since $h_{R}=\Hess z/|{\nabla z}|$ and $\mathcal{L}_{\n}g=2h_{R}$, we have 
    \be\label{eq-the derivative of induced metric}    \partial_{z}g|_{\Sigma_{R}}=\mathcal{L}_{\frac{\partial}{\partial z}}g|_{{\Sigma_{R}}}=(1+nz^{-2}+o(z^{-2}))\mathcal{L}_{\n}g|_{T{\Sigma_{R}}}=2(1+2nz^{-2}+o(z^{-2}))\Hess z.
    \ee
    Also,
    \be\label{eq-Hess z}
    \Hess z=\frac{\Hess F}{\sqrt{F}}-\frac{1}{2F^{\frac{3}{2}}}dF\otimes dF=\frac{1}{z}g-\frac{1}{z}dz\otimes dz+(\eta+o(1))z^{-3}g=\frac{1}{z}g|_{T\Sigma_{R}}+(\eta+o(1))z^{-3}g.
    \ee
    Therefore,
    \be
    \frac{\partial}{\partial z}g_{R}=\frac{-2g|_{T\Sigma_{R}}}{z^{3}}+\frac{1}{z^{2}}\frac{\partial g|_{T|_{\Sigma_{R}}}}{\partial z}=2(2n+\eta+o(1))z^{-4}g|_{T\Sigma_{R}}=2(2n+\eta+o(1))z^{-2}g_{R}.
    \ee
    This implies $|\partial_{z}g_{R}|\leq Cz^{-2}g_{R}$ for some $C>0$ and after integrating from $w$ to $z$, 
    \be\label{eq-uniform convergence of metric}
    e^{\frac{C}{z}-\frac{C}{w}}\varphi_{w-R_{0}}^{*}g_{R}|_{\Sigma_{w}}\leq \varphi_{z-R_{0}}^{*}g_{R}|_{\Sigma_{z}}\leq e^{\frac{C}{w}-\frac{C}{z}}\varphi_{w-R_{0}}^{*}g_{R}|_{\Sigma_{w}},
    \ee
    for $R_{0}\leq w\leq z$. Therefore, there is a $\mathcal{C}^{0}$-metric $g_{\infty}$ on $\Sigma_{R_{0}}$ such that $(\Sigma_{R},g_{R})$ converges to $(\Sigma_{R_{0}},g_{\infty})$ in the $\mathcal{C}^{0}$-sense via the diffeomorphism $\varphi_{R-R_{0}}:\Sigma_{R_{0}}\rightarrow\Sigma_{R}$.
\vskip0.2cm
    Now, for any $0<a<b<\infty$, denote the annulus region of the level set
    \be
    (A(a,b;t),t^{-2}g):=(z^{-1}([at,bt]), t^{-2}g).
    \ee
    Since $\|{\varphi_{s-R_{0}}^{*}g_{s}}-\varphi_{at-R_{0}}^{*}g_{at}\|_{g_{R_{0}}}=o(1)$ for $s\in[at,bt]$ as $t\rightarrow\infty$, this implies
    \be
    \lim_{t\rightarrow\infty}\|(\Phi_{t}^{-1})^{*}(t^{-2}g)-(dz\otimes dz+z^{2}(\varphi_{zt-R_{0}}^{*}g_{z}))\|_{\mathcal{C}^{0}(dz^{2}+z^{2}g_{R_{0}})}=0,
    \ee
    where $\Phi_{t}(p):=(z(p)/t,\varphi_{z(p)-R_{0}}^{-1}(p))$ is a diffeomorphism between $A(a,b;t)$ and $[a,b]\times\Sigma_{R_{0}}$ as $t\gg1$. Hence,
    \be
    \lim_{t\rightarrow\infty}\|(\Phi_{t}^{-1})^{*}(t^{-2}g)-(dz\otimes dz+z^{2}g_{\infty})\|_{\mathcal{C}^{0}(dz^{2}+z^{2}g_{R_{0}})}=0.
    \ee    
    Now, fixed $\varepsilon>0$. Since $(\frac{1}{4}-\varepsilon)d_{g}(p,\cdot)^{2}\leq F(\cdot)\leq (\frac{1}{4}+\varepsilon)d_{g}(p,\cdot)^{2}$ holds for $d_{g}(p,\cdot)\gg1$,
    \be
    A(a,b;t)\subset \left\{ (1-4\varepsilon)a^{2}\leq d_{t^{-2}g}(p,\cdot)^{2}\leq (1+4\varepsilon)b^{2}\right\}\subset A\left((1-4\varepsilon)^{2}a,(1+4\varepsilon)^{2}b;t\right),
    \ee
    for all $t\gg1$ and $0<a<b$. Letting $t\rightarrow\infty$ and $\varepsilon\rightarrow0$, we yield 
    \be
    \left(\left\{ a\leq d_{t^{-2}g}(p,\cdot)\leq b\right\},{t^{-2}g}\right)\overset{\text{GH}}{\longrightarrow}([a,b]\times \Sigma_{R_{0}},dz\otimes dz+z^{2}g_{\infty})
    \ee
    for any $0<a<b$. 
    Letting  $a\rightarrow0^{+}$, the ball $\overline{B_{t^{-2}g}(p,b)}$ converges to the cone metric $(\{o\}\cup\left((0,b]\times \Sigma_{R_{0}}\right),dz\otimes dz+z^{2}g_{\infty})$ as $t\rightarrow\infty$. Since $b>0$ is arbitrary, by the definition of pointed Gromov-Hausdorff convergence, any asymptotic cone is $C(\Sigma_{R_{0}},g_{\infty})$. The $\mathcal{C}^{0}_{\text{loc}}$-closeness outside the apex follows from the almost isometric map $\Psi_{t}:\left(\left\{ a\leq d_{t^{-2}g}(p,\cdot)\leq b\right\},{t^{-2}g}\right)\rightarrow([a-o(1),b+o(1)]\times\Sigma_{R_{0}},dz^{2}+z^{2}g_{\infty})$ as $t\rightarrow\infty$ for any $0<a\leq b<\infty$. \newline
    \vskip0.1cm
    It remains to prove the volume non-collapsing. By (\ref{eq-uniform convergence of metric}), there is a constant $v(\Sigma_{R_{0}})>0$ such that 
    \be
    \frac{\vol_{\varphi_{z-R_{0}}^{*}g_{R}|_{\Sigma_{z}}}(x,r)}{r^{n-1}}\geq v,
    \ee
    for all $x\in\Sigma_{R_{0}}$, $r\in(0,1]$, and $z\geq R_{0}$. Passing $z\rightarrow\infty$, (\ref{eq-uniform convergence of metric}) implies
    \be
    \frac{\vol_{g_{\infty}}(x,2r)}{r^{n-1}}\geq\frac{\vol_{\varphi_{z-R_{0}}^{*}g_{R}|_{\Sigma_{z}}}(x,r)}{(2r)^{n-1}}\geq \frac{v}{2^{n-1}},
    \ee
    for all $x\in\Sigma_{R_{0}}$ and $r\in(0,1]$. Therefore, since 
    \be
    d_{dz^{2}+z^{2}g_{\infty}}((R_{1},x_{1}),(R_{2},x_{2}))^{2}=\begin{cases}
        R_{1}^{2}+R_{2}^{2}-2R_{1}R_{2}\cos(d_{g_{\infty}}(x_{1},x_{2})),\text{ if }d_{g_{\infty}}(x_{1},x_{2})\leq\pi;\\
        (R_{1}+R_{2})^{2},\text{ otherwise.}
    \end{cases}
    \ee
    for all $(R_{i},x_{i})\in C(\Sigma_{R_{0}})$, we have
    \be
    \begin{aligned}
    {\vol_{dz^{2}+z^{2}g_{\infty}}((R,x),r)}=&~\int_{R-r}^{R+r}z^{n-1}\vol_{g_{\infty}}\left(x,\arccos\left(\frac{-r^{2}+R^{2}+z^{2}}{2Rz}\right)\right)dz\\
    \geq&~\int_{R-r}^{R+r}z^{n-1}\vol_{g_{\infty}}\left(x,\frac{|z-R|}{z}\right)dz\\
    \geq&~\frac{v}{n2^{2n-2}}r^{n},
    \end{aligned}
    \ee
    for $r\leq\min\{1,\pi,\frac{R}{2}\}$. Therefore, by the $\mathcal{C}^{0}_{\text{loc}}$-convergence via $\Psi_{t}$, for $r(x)=:R\gg1$, we have $2\Psi_{R}^{*}(dz^{2}+z^{2}g_{\infty})\geq R^{-2}g\geq\frac{1}{2}\Psi_{R}^{*}(dz^{2}+z^{2}g_{\infty})$ on $r(z)\in[\frac{R}{4},4R]$. Hence,
    \be
    \begin{aligned}
        \vol_{g}(x,1)=\frac{\vol_{R^{-2}g}(x,R^{-1})}{R^{-n}}\geq\frac{\vol_{dz^{2}+z^{2}g_{\infty}}(\Psi_{R}(x),1/2R)}{2^{n}R^{-n}}\geq\frac{v}{2^{4n-2}n}.
    \end{aligned}
    \ee
    This confirms the volume non-collapsing for all $x\in\M$.
\end{proof}
\begin{proof}[Proof of Theorem \ref{thm-expander with quadratic decay}
]
    We first show that $|{\Rm}|$ decays at least sub-quadratically. By Proposition \ref{int-estimate-expander}, for $p>a+c$ and $a>2n$,
    \be
    \int_{B_{1}(x)}|{\Rm}|^{p}\leq C'(p)(r(x)+1)^{-a}
    \ee
    for $r(x)\gg1$. By the calculation in \cite{CLX23}, the curvature satisfies
    \be
    \Delta|{\Rm}|^{2}\geq-C(F+|{\Rm}|)|{\Rm}|^{2}.
    \ee    
    Denote $u:=F+|{\Rm}|$. Now, we apply Moser's iteration argument (\cite{Li93}) to get
    \be\label{eq-Moser iteration}
    |{\Rm}|(x)^{2}\leq \tilde{C}\left[\left(\frac{\int_{B_{1}(x)}u^{n}}{\vol(x,1)}\right)^{\frac{2}{n}}+1\right]^{\frac{n}{p}}\left(\frac{\int_{B_{1}(x)}|{\Rm}|^{p}}{\vol(x,1)}\right)^{\frac{2}{p}},
    \ee
    for all $x\in\M$, where $\tilde{C}$ depends on $p,n$ and the lower bounds of $\Ric$ on $B_{1}(x)$ (Sobolev constant depends on the lower bound of Ricci curvature). Since the Ricci curvature is bounded, $\tilde{C}$ is a constant. By Proposition \ref{prop-vol lower bound of expander}, there is a constant $v>0$ such that $\vol(x,1)\geq v$ for all $x\in\M$. Also, H\"older inequality and the volume comparison imply
    \be
    \int_{B_{1}(x)}|{\Rm}|^{n}\leq\left(\int_{B_{1}(x)}|{\Rm}|^{p}\right)^{\frac{n}{p}}\vol(x,1)^{\frac{p-n}{p}}\leq C(1+r(x))^{-\frac{na}{p}}.    
    \ee
    Therefore, 
    \be \label{estimate}
    \frac{\int_{B_{1}(x)}u^{n}}{\vol(x,1)}\leq c\frac{\int_{B_{1}(x)}|{\Rm}|^{n}}{\vol(x,1)}+c(1+r(x))^{2n}\leq c(1+r(x))^{2n}.
    \ee
    Substitute (\ref{estimate}) into (\ref{eq-Moser iteration}),
    \be
    |{\Rm}|(x)^{2}\leq\tilde{C}(1+r(x))^{-\frac{2a}{p}+\frac{4n}{p}}=\tilde{C}(1+r(x))^{\frac{4n-2a}{p}}.
    \ee
    This implies that the curvature tends to zero at infinity if we choose $a=2n+1$.

    Now, we estimate the distance along the level set flow of $F$ as in \cite[Lemma 7]{C23}. Define $\Pi_{r}:=\{F=r\}$ for all $r\geq0$. Since $|\nabla F|^{2}=F-\scal-\frac{n}{2}$ and $\scal\rightarrow0$ as $x\rightarrow\infty$, may assume $|\nabla F|>.5$ on $\{F\geq\rho_{0}\}$. Then for any $x\in\Pi_{R}$ with $R\geq\rho_{0}$, there is a unique flow $\gamma$ generated by $\nabla F/|\nabla F|^2$, starting from $\gamma(\rho_{0})=:z\in\Pi_{\rho_{0}}$, and passing through $x=\gamma(R)$. Since $|\nabla F|>.5$ on $\{F\geq\rho_{0}\}$, the flow can be extended to infinity. Also, $F(\gamma(t))=t$ for all $t\geq\rho_{0}$. We claim the following: 
    \begin{clm}\label{t-simeq-distance-square}
        There is a constant $C(\Pi_{\rho_{0}},p)>0$ such that \be \frac{1}{4}r(\gamma(t))^{2}-Cr(\gamma(t))-C\leq t\leq \frac{1}{4}r(\gamma(t))^{2}+Cr(\gamma(t))+C
        \ee
        for all $t\geq\rho_{0}$ and all $z\in\Sigma_{\rho}$.
    \end{clm}
    \begin{proof}[Proof of Claim \ref{t-simeq-distance-square}]
    Lower and upper bounds all follow from \cite[Lemma 2.2]{CD15}. More precisely, there is a constant $C(p,\M)>0$ such that
    \be
    t=F(\gamma(t))\geq\frac{1}{4}r(\gamma(t))^{2}-Cr(\gamma(t))-C,
    \ee
    and 
    \be
    t=F(\gamma(t))\leq\frac{1}{4}(r(\gamma(t))+C)^{2}+C
    \ee
    for all $t\geq0$ and $z\in\Pi_{\rho_{0}}$.
    \end{proof}
    For expanding solitons, the Bianchi identity implies
    \be
    \nabla_{j}\Rm_{ilkj}=\nabla_{i}\Ric_{kl}-\nabla_{l}\Ric_{ki}=\Rm_{ilkj}f_{j}
    \ee
    Therefore,
    \be\label{curvature-along-flow}
    \begin{aligned}
        \frac{d}{dt}|{\Rm}|^{2}(\gamma(t))=&~2R_{ijk\ell}R_{ijk\ell,m}\frac{F_{m}}{|{\nabla F}|^{2}}\\
        =&~\frac{2R_{ijkl}}{|\nabla F|^2{}}\left[(R_{ij\ell m}F_{m})_{,k}+(R_{ijmk} F_{m})_{,\ell}-R_{ij\ell m}F_{mk}-R_{ijmk}F_{m\ell}\right]\\
        \leq&~8\frac{|{\Rm}||{\nabla^{2}\Ric}|}{|\nabla F|^{2}}+4\frac{|{\Rm}|^{2}|{\Ric}|}{|\nabla F|^{2}}-2\frac{|{\Rm}|^{2}}{|\nabla F|^{2}},
    \end{aligned}
    \ee
    where we use the soliton equation in the last line. Since we have bounded curvature now, by \cite[Lemma 2.8]{D17} and Claim \ref{t-simeq-distance-square},
    \be
    |{\nabla^2\Ric}|(\gamma(t))\leq\frac{C}{t},
    \ee
    for $t\geq\rho_{0}+C'$, where the constant $C'$ is determined by the constant in Claim \ref{t-simeq-distance-square}). Also, by Claim \ref{t-simeq-distance-square} and the identity $|\nabla F|^{2}+\scal=F$, we have $\frac{t}{2}-C''\leq|\nabla F|\leq \frac{t}{2}+C''$ for $t\geq\rho_{0}+C'$. Thus (\ref{curvature-along-flow}) becomes
    \be
    \frac{d}{dt}|{\Rm}|(\gamma(t))\leq Ct^{-2}+Ct^{-2}|{\Rm}|-\frac{|{\Rm}|}{t+C}\leq-\frac{1}{t+C}|{\Rm}|+\frac{2C}{t^{2}},
    \ee
    for $t\geq\rho_{0}+C'$ since $|{\Rm}|$ is bounded. This implies 
    \be
    \left((t+C)|{\Rm}|(\gamma(t))\right)'\leq\frac{C}{t},
    \ee
    for $t\geq\rho_{0}+C'$. Then integrating it from $\rho_{0}+C'$ to $t$, we get $|{\Rm}|(\gamma(t))\leq C\log t/t$ for $t\geq \rho_{0}+C'$. Combining this with Claim \ref{t-simeq-distance-square}, we obtain the desired bound $\BigO(r^{-2}\log r)$. 
    \vskip0.2cm
    For $n=4$, by \cite[Proposition 3.1]{CL22}, we have
    \be
    |{\Rm}|\leq C_{0}\left(|{\Ric}|+\frac{|{\nabla\Ric}|}{|{\nabla f}|}\right),
    \ee
    on $\M$ for some universal constant $C_{0}>0$ (see also \cite{MW15}). Since $|{\Ric}|=\BigO(r^{-2})$ and $|{\Rm}|=\BigO(1)$ as $r\rightarrow\infty$, by the local Shi's estimate \cite[Lemma 2.8]{D17},
    \be
    |{\nabla\Ric}|\leq\frac{C}{r^{2}},
    \ee
    as $r\gg1$. Also, $|{\nabla f}|\geq r/4$ as $r\gg1$, this implies
    \be
    \limsup_{r\rightarrow\infty}r^{2}|{\Rm}|\leq C_{0}\limsup_{r\rightarrow\infty}\left(r^2|{\Ric}|+\frac{r^2|{\nabla\Ric}|}{|{\nabla f}|}\right)=C_{0}\eta<+\infty,
    \ee
    which implies the required estimate.
\end{proof}

\begin{proof}[Proof of Corollary \ref{cl-topology of ends on expander}]
    In Proposition \ref{prop-vol lower bound of expander}, we proved the $\mathcal{C}^{0}$-convergence and yield a $\mathcal{C}^{0}$-metric $g_{\infty}$ on $\Sigma_{R_{0}}$. We claim that the Ricci curvature on $(\Sigma_{R},R^{-2}g|_{\Sigma})$ is uniformly bounded.
    \begin{clm}\label{cl2-Almost Einstein}
        (Almost $\eta$-Einstein)
            \be
                \limsup_{r\rightarrow\infty}\left\|{\Ric(r^{-2}g|_{\Sigma_{r}})-(n-2)r^{-2}g|_{\Sigma_{r}}}\right\|_{r^{-2}g|_{\Sigma_{r}}}\leq\eta.
            \ee
    \end{clm}
    \begin{proof}[Proof of Claim \ref{cl2-Almost Einstein}]
    By the Gauss equation for the Ricci tensor,
        \be\label{eq-cl2-Gauss eq}
        \Ric(g|_{\Sigma_{Rr}})(X,X)=\Ric(g)(X,X)-\Rm(g)(X,\n,\n,X)-\sum_{i=1}^{n-1}\det h_{R}(X,E_{i}),
        \ee
        where $X\in T\Sigma_{R}$ and $\{E_{i}\}_{i=1}^{n-1}$ is an orthonormal basis of $T\Sigma_{R}$ with respect to $g$. By the Bianchi identity for expanding soliton, 
        \be
        \Rm(g)(X,\n,\n,X)=\frac{\nabla_{X}\Ric(g)(\n,X)-\nabla_{\n}\Ric(g)(X,X)}{|{\nabla F}|}.
        \ee
        Therefore, 
        \be
        |{\Rm(g)(\cdot,\n,\n,\cdot)}|\leq2\frac{|{\nabla\Ric(g)}|_{g}}{|\nabla F|}.
        \ee
        By Theorem \ref{thm-expander with quadratic decay}, $\Rm(g)$ is uniformly bounded on $\M$. Hence, by the local Shi's estimate \cite[Lemma 2.8]{D17} on $B(x,2\delta r(x))$ with $\delta\ll1$, we have
        \be
        \sup_{B(x,\delta r(x))}|{\nabla\Ric(g)}|\leq C\left(\sup_{A(x;\delta r(x),2\delta r(x))}\frac{|{\nabla F}|}{\delta r(x)},|{\Rm}|, n\right)\sup_{B(x,2\delta r(x))}|{\Ric}|\leq\frac{C}{r(x)^{2}},
        \ee
        for $r(x)\gg1$ and some constant $C>0$. Combining with $|{\nabla F}|=\frac{r(x)}{2}+\BigO(\sqrt{r(x)})$ and $r(x)=R+\BigO(1)$, we get
        \be\label{eq-cl2-normal direction}
        |{\Rm(g)(\cdot,\n,\n,\cdot)}|\leq\frac{C}{R^{3}},
        \ee
        for $R>R_{0}$ with some constant $R_{0}$ depending on the above constant $C$. Also, by (\ref{eq-Hess z}),
        \be
        h_{R}=\frac{\Hess z}{|{\nabla z}|}=\frac{1}{R}g|_{\Sigma_{R}}+\BigO(R^{-3}),
        \ee
        as $R\rightarrow\infty$. Therefore, 
        \be\label{eq-cl2-h_R}
        \begin{aligned}
            \sum_{i=1}^{n-1}\det h_{R}(X,E_{i})=&~\sum_{i=1}^{n-1}h_{R}(X,E_{i})^{2}-H_{R}h_{R}(X,X)\\
            =&~R^{-2}g(X,X)-(n-1)R^{-1}R^{-1}g(X,X)+\BigO(R^{-4})\\
            =&~(2-n)R^{-2}g(X,X)+\BigO(R^{-4}),
        \end{aligned}
        \ee
        as $R\rightarrow\infty$ for $X\in T\Sigma_{R}$. Combining (\ref{eq-cl2-Gauss eq}), (\ref{eq-cl2-normal direction}), (\ref{eq-cl2-h_R}), and the assumption on Ricci curvature, 
        \be
        \limsup_{R\rightarrow\infty}\left\|{\Ric(R^{-2}g|_{\Sigma_{R}})}-(n-2)R^{-2}g\right\|_{R^{-2}g}\leq\limsup_{R\rightarrow\infty}R^{2}\|\Ric(g)\|_{g}=\eta.
        \ee
        This completes the proof.
    \end{proof}
\end{proof} 
\begin{proof}[Proof of Corollary \ref{cl-structure of limit of fiber}]  By Claim \ref{cl-structure on fiber C0} and Claim \ref{cl2-Almost Einstein}, $(\Sigma_{R_{0}},d_{g_{\infty}})$ is a non-collapsed two-sided Ricci limit space. Note that by (\ref{eq-uniform convergence of metric}), for any $\eta,\epsilon>0$ and $x\in\Sigma_{R_{0}}$, there are constants $R_{\eta}\geq R_{0}$ and $r_{x,\eta,\epsilon}\in(0,1)$ such that 
\be
(1-\eta)^{2}s^{-2}\varphi_{s-R_{0}}^{*}g|_{\Sigma_{s}}\leq t^{-2}\varphi_{t-R_{0}}^{*}g|_{\Sigma_{t}}\leq(1+\eta)^{2}s^{-2}\varphi_{s-R_{0}}^{*}g|_{\Sigma_{s}}
\ee
for all $t\geq s\geq R_{\eta}$ and \be
    \frac{\vol_{R_{\eta}^{-2}\varphi_{R_{\eta}-R_{0}}^{*}g|_{\Sigma_{R_{\eta}}}}(x,r)}{\omega_{n-1}r^{n-1}}\in(1-\epsilon,1+\epsilon),
\ee
for all $r\in(0,r_{x,\eta,\epsilon})$. Therefore, for $t\geq R_{\eta}$, we have
\begin{eqnarray*}
    \frac{\vol_{t^{-2}\varphi_{t-R_{0}}^{*}g|_{\Sigma_{t}}}(x,(1+\eta)r)}{\omega_{n-1}r^{n-1}}&\geq&\frac{(1-\eta)^{n-1}\vol_{R_{\eta}^{-2}\varphi_{R_{\eta}-R_{0}}^{*}g|_{\Sigma_{R_{\eta}}}}(x,r)}{\omega_{n-1}r^{n-1}}\\
    &\geq&(1-\eta)^{n-1}(1-\varepsilon)\\
\Rightarrow~\frac{\vol_{t^{-2}\varphi_{t-R_{0}}^{*}g|_{\Sigma_{t}}}(x,\hat{r})}{\omega_{n-1}\hat{r}^{n-1}}&\geq& (1-\eta^{2})^{n-1}(1-\epsilon),
\end{eqnarray*}
for all $\hat{r}\in(0,\frac{r_{x,\eta,\epsilon}}{1+\eta})$. On the other hands, by Corollary \ref{cl-topology of ends on expander}, $|\Ric(t^{-2}\varphi_{t-R_{0}}^{*}g|_{\Sigma_{t}})|\leq n-2+\eta$ has a uniform bound for $t\geq R_{0}$. Thus, by \cite[Remark 3.3]{A90} (see also \cite[Theorem 2.1]{H19}) and set $\eta,\epsilon$ small enough so that $(1-\eta^{2})^{n-1}(1-\epsilon)\geq 1-\varepsilon(n-1,n-2+\eta)$, where $\varepsilon(n-1,n-2+\eta)$ is the constant in \cite[Remark 3.3]{A90}, the metric $t^{-2}\varphi_{t-R_{0}}^{*}g|_{\Sigma_{t}}$ converge in the $\mathcal{C}^{1,\alpha}$-sense to $g_{\infty}$ on $B_{g_{\infty}}(x,\tilde{r})$ for some $\tilde{r}\in(0,\frac{r_{x,\eta,\epsilon}}{1+\eta})$. Since $x\in\Sigma_{R_{0}}$ is arbitrary and $\Sigma_{R_{0}}$ is compact, we yield a $\mathcal{C}^{1,\alpha}$-convergence on $\Sigma_{R_{0}}$ and the limiting metric $g_{\infty}$ is $\mathcal{C}^{1,\alpha}$ for all $\alpha\in(0,1)$.
    \newline
\vskip0.1cm
   It remains to show the $\mathcal{C}^{1,\alpha}_{\text{loc}}$-convergence on $C(\Sigma_{R_{0}})\setminus\{o\}$. For $x\in\Sigma_{R}$, by Colding's volume convergence theorem \cite{C97}, for any $\varepsilon>0$, there is an $r_{x,\varepsilon}>0$ such that $\vol_{g_{\infty}}(x,r)\geq(1-\varepsilon)\omega_{n-1}r^{n-1}$ for $r\in[0,r_{x,\varepsilon}]$. Then for any $T>0$, there is a $\tilde{r}(x,\varepsilon,T)\ll r_{x,\epsilon}$ such that $\vol_{dz^{2}+z^{2}g_{\infty}}((T,x),r)\geq(1-\varepsilon'(\varepsilon))\omega_{n}r^{n}$ for all $r\leq\tilde{r}$, $\varepsilon'$ is a positive constant such that $\lim_{\varepsilon\rightarrow0}\varepsilon'=0$. By the $\mathcal{C}^{0}$-convergence via $\Psi_{t}$, this implies $\vol_{(\Psi_{t}^{-1})^{*}t^{-2}g}(T,x)\geq(1-\varepsilon'(\varepsilon))\omega_{n}r^{n}$ for all $r\leq \tilde{r}$ and $t\geq R(x,T,\varepsilon)$. Now, by \cite[Remark 3.3]{A90}, after choosing $\varepsilon':=\epsilon(n,T)$ where $\epsilon$ is the constant in \cite[Remark 3.3]{A90}, the $\mathcal{C}^{1,\alpha}$-norm at $(x,T)$ with respect to the metric $(\Psi_{t}^{-1})^{*}t^{-2}g$ is uniformly bounded for $t\gg1$. Since the $\mathcal{C}^{0}$-limit is unique and $\alpha\in(0,1)$ is arbitrary, these imply the uniform $\mathcal{C}^{1,\alpha}_{\text{loc}}$-convergence on $C(\Sigma_{R_{0}})\setminus\{o\}$ for any $\alpha\in(0,1)$.
    \end{proof}
\section{Main Theorems - steady soliton}\label{sec-proof of steady case}
Thought out this section, we denote $z:=F:=-f$, the level set $\Sigma_{R}=\{F=R\}$, and the flow $\frac{d}{dt}\varphi_{t}=\frac{\nabla F}{|{\nabla F}|^{2}}$ on the region $|{\nabla F}|>0$. Set $R_{0}>0$ with $|{\nabla F}|>.5$ on $\{F\geq R_{0}-4\}$. Also, by Proposition \ref{F is approximately r} $F=r(x)+\BigO(\log r(x))$, $|{\nabla F}|=\sqrt{1-\scal}=1+\BigO(r^{-1})$, and $\Hess F=\BigO(r^{-1})$. Let $g_{R}=g|_{\Sigma_{R}}$ be the induced metric on $\Sigma_{R}$ and $\n=\frac{\nabla F}{|{\nabla F}|}$ be the normal vector on $\Sigma_{R}$. Since $\mathcal{L}_{\n}g=2h_{R}$, where $h_{R}=\frac{\Hess F}{|{\nabla F}|}$
\be\label{eq-evolution of g|R- steady}
\frac{\partial}{\partial z}g_{R}=\frac{1}{|{\nabla F}|}\mathcal{L}_{\n}g=\frac{2}{|{\nabla F}|}h_{R}=\frac{2\Hess F}{|{\nabla F}|^{2}}=\BigO(z^{-1})g_{R},
\ee
for $z\geq R_{0}$. This implies that there exists a constant $C>0$ such that 
\be
    \frac{\partial}{\partial z} (z^{C}g_{R})\geq0\geq \frac{\partial}{\partial z}(z^{-C}g_{R}),
\ee
for $z\geq R_{0}$. Using the diffeomorphism $\varphi_{R-R_{0}}:\Sigma_{R_{0}}\rightarrow\Sigma_{R}$ to pullback the metric $g_{R}$, 
\be
    \left(\frac{s}{t}\right)^{C}\varphi_{s-R_{0}}^{*}g|_{\Sigma_{s}}\leq\varphi_{t-R_{0}}^{*}g|_{\Sigma_{t}}\leq\left(\frac{t}{s}\right)^{C}\varphi_{s-R_{0}}^{*}g|_{\Sigma_{s}},
\ee
on $\Sigma_{R_{0}}$, for all $t\geq s\geq R_{0}$. Also, $g(\frac{\partial}{\partial z},\frac{\partial}{\partial z})=|{\nabla F}|^{-2}\in[1,4]$. Therefore, for $T>T(R_{0},C)$ large enough, on the annulus region $z^{-1}[T+R_{0},T+R_{0}+2]$, we have
\be
\left(\frac{R_{0}}{T+R_{0}+2}\right)^{C}g\leq\varphi_{T}^{*}g\leq\left(\frac{T+R_{0}+2}{R_{0}}\right)^{C}g,
\ee
on $z^{-1}[R_{0},R_{0}+2]$. Choose $v>0$ such that $\frac{\vol_{g}(x,r)}{r^{n}}\geq v$ for all $x\in\Sigma_{R_{0}+1}$ and $r\in[0,1]$. Hence, for $z(x)=T+R_{0}+1\gg1$,
\be
\begin{aligned}
    \vol_{g}(x,1)=\vol_{\varphi_{T}^{*}g}(\varphi_{-T}(x),1)\geq&~\left(\frac{R_{0}}{T+R_{0}+2}\right)^{\frac{n}{2}C}\vol_{g}\left(\varphi_{-T}(x),\left(\frac{R_{0}}{T+R_{0}+2}\right)^{\frac{C}{2}}\right)\\
    \geq&~\left(\frac{R_{0}}{T+R_{0}+2}\right)^{nC}v\\
    \geq&~(1+r(x))^{-nC}\tilde{v},
\end{aligned}
\ee
where $\tilde{v}$ is a positive constant independent of $x$.

In summary, we have proved the following lemma.
\begin{lm}\label{unit volume for steady soliton}
    Let $(\M,g,f)$ be a complete non-Ricci flat gradient steady Ricci soliton. Suppose that $f$ is proper and the Ricci curvature decays linearly. Then there are two constants $c,N>0$ and $x_{0}\in\M$ such that 
    \be
    \vol(x,1)\geq c(1+d_{g}(x,x_{0}))^{-N},
    \ee
    for all $x\in\M$.
\end{lm}
\begin{proof}[Proof of Theorem \ref{thm-steady with linear decay}]
    The proof is the same as the first part of the proof of Theorem \ref{thm-expander with quadratic decay}. By Theorem \ref{int-estimate}, there are $a,p>0$ such that
    \be
    \int_{\M}|{\Rm}|^{p}F^{-a}<\infty.
    \ee
    (For instance, we can take $a>cp^{5}$ and $p\geq n$.) Then this implies 
    \be\label{p-norm-curvature}
        \int_{B_{1}(x)}|{\Rm}|^{p}\leq C(1+r(x))^{a}
    \ee
    for all $x\in\M$. Also, by (\ref{Laplacian-Rm}), the evolution equation of $|{\Rm}|^{2}$ becomes
    \be
    \begin{aligned}
        \Delta|{\Rm}|^{2}\geq&~2|{\nabla\Rm}|^{2}-\langle\nabla F,\nabla|{\Rm}|^{2}\rangle-c|{\Rm}|^{3}\\
        \geq&~2|{\nabla\Rm}|^{2}-2|{\nabla F}||{\Rm}||{\nabla\Rm}|-c|{\Rm}|^{3}\\
        \geq&~2|{\nabla\Rm}|^{2}-c|{\nabla F}|^{2}|{\Rm}|^{2}-\frac{1}{4}|{\nabla\Rm}|^{2}-c|{\Rm}|^{3}\\
        \geq&-c(|{\Rm}|+1)|{\Rm}|^{2}.
    \end{aligned}
    \ee
    Now we apply the Moser's iteration again, 
    \be\label{Moser-iteration}
    |{\Rm}|(x)^{2}\leq \tilde{C}\left[\left(\frac{\int_{B_{1}(x)}|{\Rm}|^{p}}{\vol(x,1)}\right)^{\frac{1}{p-\frac{n}{2}}}+1\right]^{\frac{n}{p}}\left(\frac{\int_{B_{1}(x)}|{\Rm}|^{p}}{\vol(x,1)}\right)^{\frac{2}{p}}
    \ee
    where $\tilde{C}$ depends on $n,p$ and the Sobolev constant on $B_{1}(x)$. Since the Ricci curvature is uniformly bounded on $\M$, $\tilde{C}$ is a constant. Now, by Lemma \ref{unit volume for steady soliton}, we have $\vol(x,1)\geq c(1+r(x))^{-N}$ for some constants $c,N>0$. Plugging this and (\ref{p-norm-curvature}) into (\ref{Moser-iteration}), it implies a polynomial bound of $|{\Rm}|$. This completes the proof.
\end{proof}
\begin{proof}[Proof of Corollary \ref{cl-of-bdd-curvature}]
    Since $f$ is proper and (\ref{normalized cond}) holds, we may assume $|\nabla F|\geq1/2$ on $\{F\geq\rho_{0}\}$. Define $\Sigma_{r}:=\{F=r\}$ for all $r\in\R$. Then for $R\geq\rho_{0}$ and $x\in\Sigma_{R}$, there is the flow $\gamma$ of the vector field $\frac{\nabla F}{|\nabla F|^{2}}$ starting from $\gamma(\rho_{0}):=y\in\Sigma_{\rho_{0}}$ to $\gamma(R):=x$. Note that
    \be\label{curvature-along-flow-steady}
    \begin{aligned}
        \frac{d}{dt}|{\Rm}|(\gamma(t))^{2}=&~2R_{ijk\ell}R_{ijk\ell,m}\frac{F_{m}}{|\nabla F|^{2}}\\
        =&~\frac{2R_{ijkl}}{|\nabla F|^2{}}\left[(R_{ij\ell m}F_{m})_{,k}+(R_{ijmk} F_{m})_{,\ell}-R_{ij\ell m}F_{mk}-R_{ijmk}F_{m\ell}\right]\\
        \leq&~8\frac{|{\Rm}||{\nabla^{2}\Ric}|}{|\nabla F|^{2}}+4\frac{|{\Rm}|^{2}|{\Ric}|}{|\nabla F|^{2}}
    \end{aligned}
    \ee
    where we apply the Bianchi identity at the second line and (\ref{nable-Ric}) at the last line. Since $|\nabla F|\leq1$, 
    \be
    t-\rho_{0}=F(\gamma(t))-F(\gamma(\rho_{0}))\leq d(y,\gamma(t))\leq d(x_{0},\Sigma_{\rho_{0}})+r(\gamma(t)).
    \ee
    On the other hand, 
    \be
    \begin{aligned}
        r(\gamma(t))-d(x_{0},\Sigma_{\rho_{0}})\leq&~ d(y,\gamma(t))\\
        \leq&~\int_{\rho_{0}}^{t}\frac{1}{|\nabla F|}(\gamma(w))dw\\
        \leq&~2(t-\rho_{0}).
    \end{aligned}
    \ee
    Therefore, there is a constant $C(\rho_{0},x_{0},\Sigma_{\rho_{0}})>0$ such that
    \be\label{t-sim-distance}
        \frac{1}{2}r(\gamma(t))-C\leq t\leq r(\gamma(t))+C,
    \ee
    for $t\geq\rho_{0}$. 
        
    Now under our assumption, for any $p>3$, there is a $r_{0}>0$ such that 
    \be
    |{\Ric}|\leq\frac{1}{p^{5}F},
    \ee
    outside $D(r_{0})$. Therefore, by Theorem \ref{int-estimate}, for  $\delta=p^{-5}$ and $a=N+c$, which is a constant independent of $p$, 
    \be
    \int_{B_{1}(z)}|{\Rm}|^{p}\leq C_{p}(1+r(z))^{a}
    \ee
    for some constant $C_{p}>0$. As the proof of Theorem \ref{thm-steady with linear decay}, we adopt Moser's iteration and yield
    \be
    |{\Rm}|(z)^{2}\leq C_{p}'(1+r(z))^{\frac{4(a+N)}{p}},
    \ee
    for all $z\in\M$ and for some constant $C_{p}'>0$. Take $p=\frac{8(a+N)}{\varepsilon}$. Then
    \be\label{poly-growth-curvature}
    |{\Rm}|(z)\leq C'(1+r(z))^{\frac{\varepsilon}{2}},
    \ee
    for all $z\in\M$, where $\varepsilon>0$ is the positive constant in (\ref{Ricfast}). By assumption (\ref{Ricfast}), there is a constant $C''>0$ such that 
    \be\label{Ricci-decay-rapidly}
    |{\Ric}|(z)\leq C''(1+r(z))^{-1-\varepsilon},
    \ee
    for all $z\in\M$. Combining (\ref{curvature-along-flow-steady}), (\ref{t-sim-distance}), (\ref{poly-growth-curvature}), (\ref{Ricci-decay-rapidly}), and the local Shi's estimate \cite[Remark 2.7, Lemma 2.8]{D17},
    \be
    \frac{d}{dt}|{\Rm}|(\gamma(t))\leq\frac{\tilde{C}}{t^{1+\frac{\varepsilon}{2}}},
    \ee
    for $t\geq\rho_{0}+C'''$, where $\tilde{C}$ is a constant independent of $\gamma$ and $C'''$ depends on the constant in (\ref{t-sim-distance}). Integrate it from $t=\rho_{0}+C'''$ to $t=R$, we get 
    \be
    |{\Rm}|(x)\leq \sup_{\Sigma_{\rho_{0}+C'''}}|{\Rm}| +\frac{2\tilde{C}}{\varepsilon}(\rho_{0}+C''')^{-\frac{\varepsilon}{2}}=:\hat{C}<\infty,
    \ee
    where $\hat{C}$ is a finite positive constant. This shows the boundedness of $|{\Rm}|$.

    Now if in addition $|{\Ric}|\leq C_0 \scal$ for some positive constant $C_0$. Then the scalar curvature $\scal=O(r^{-1-\varepsilon})$ satisfies
    \[
    \Delta_f \scal=-2|{\Ric}|^2\ge -2C_0^2\scal^2.
    \]
    The exponential decay of $|{\Ric}|$ then follows from \cite[Proposition 1]{C19} and our assumption $|{\Ric}|\leq C_0 \scal$ on $\M$.
\end{proof}
\begin{rmk}
    Based on the proof of Corollary \ref{cl-of-bdd-curvature}, one can generalize the condition $|{\Ric}|=o(r^{-1})$ and show that $|{\Rm}|=\BigO(r^{\epsilon})$ for all $\epsilon>0$ by choosing $\delta=p^{-5}$ and letting $p\rightarrow\infty$.
\end{rmk}
\begin{proof}[Proof of Corollary \ref{cl-asymptotic limit of super-linear decay}]
    By (\ref{eq-evolution of g|R- steady}) and $|{\Ric}|=\BigO(z^{-1-\varepsilon})$,
    \be
   \left|\frac{\partial}{\partial z}g_{R}\right|\leq\frac{C}{z^{1+\varepsilon}},
    \ee
    for $z\geq R_{0}$. This implies
    \be
    \exp\left(Ct^{-\varepsilon}-Cs^{-\varepsilon}\right)\varphi_{s-R_{0}}^{*}g_{s}\leq\varphi_{t-R_{0}}^{*}g_{t}\leq\exp\left(Cs^{-\varepsilon}-Ct^{-\varepsilon}\right)\varphi_{s-R_{0}}^{*}g_{s},
    \ee
    for all $t\geq s\geq R_{0}$, and hence $\varphi_{R-R_{0}}^{*}g_{R}$ uniformly converges to some continuous Riemannian metric $g_{\infty}$ on $\Sigma_{R_{0}}$. 

    \vskip0.2cm
    Now, we claim that $|{\Ric(\varphi_{R-R_{0}}^{*}g_{R})}|\rightarrow0$ as $R\rightarrow\infty$. By the Gauss equation for the Ricci tensor,
        \be\label{eq-cl2-Gauss eq-steady}
        \Ric(g|_{\Sigma_{Rr}})(X,X)=\Ric(g)(X,X)-\Rm(g)(X,\n,\n,X)-\sum_{i=1}^{n-1}\det h_{R}(X,E_{i}),
        \ee
        where $X\in T\Sigma_{R}$ and $\{E_{i}\}_{i=1}^{n-1}$ is an orthonormal basis of $T\Sigma_{R}$ with respect to $g$. By the Bianchi identity for expanding soliton, 
        \be
        \Rm(g)(X,\n,\n,X)=\frac{\nabla_{X}\Ric(g)(\n,X)-\nabla_{\n}\Ric(g)(X,X)}{|{\nabla F}|}.
        \ee
        Therefore, 
        \be
        |{\Rm(g)(\cdot,\n,\n,\cdot)}|\leq2\frac{|{\nabla\Ric(g)}|_{g}}{|\nabla F|}.
        \ee
        By Corollary \ref{cl-of-bdd-curvature}, $\Rm(g)$ is uniformly bounded on $\M$. Hence, by the local Shi's estimate \cite[Lemma 2.8]{D17} on $B(x,2\delta r(x))$ with $\delta\ll1$, we have
        \be
        \sup_{B(x,\delta r(x))}|{\nabla\Ric(g)}|\leq C\left(\sup_{A(x;\delta r(x),2\delta r(x))}\frac{|{\nabla F}|}{\delta r(x)},|{\Rm}|, n\right)\sup_{B(x,2\delta r(x))}|{\Ric}|\leq\frac{C}{r(x)^{1+\varepsilon}},
        \ee
        for $r(x)\gg1$ and some constant $C>0$. Combining with $|{\nabla F}|=1+o(1)$,
        \be\label{eq-cl2-normal direction-steady}
        |{\Rm(g)(\cdot,\n,\n,\cdot)}|\leq\frac{C}{R^{1+\varepsilon}},
        \ee
        for $R>R_{0}$ with some constant $R_{0}$ depending on the above constant $C$. Also, 
        \be
        h_{R}=\frac{\Hess F}{|{\nabla F}|}=\BigO(R^{-1-\varepsilon}),
        \ee
        as $R\rightarrow\infty$. Therefore, 
        \be\label{eq-cl2-h_R-steady}
        \begin{aligned}
            \sum_{i=1}^{n-1}\det h_{R}(X,E_{i})=&~\sum_{i=1}^{n-1}h_{R}(X,E_{i})^{2}-H_{R}h_{R}(X,X)\\
            =&~\BigO(R^{-2-2\varepsilon})g(X,X),
        \end{aligned}
        \ee
        as $R\rightarrow\infty$ for $X\in T\Sigma_{R}$. Combining (\ref{eq-cl2-Gauss eq-steady}), (\ref{eq-cl2-normal direction-steady}), (\ref{eq-cl2-h_R-steady}), and the assumption on Ricci curvature, 
        \be
        \limsup_{R\rightarrow\infty}\left\|{\Ric(g_{R})}\right\|_{g_{R}}=0.
        \ee
        Also, by Gauss equation again for sectional curvature,
        \be\label{eq-Gauss eq of Rm}
            \Rm(g|_{\Sigma{R}})(X,Y,Y,X)=\Rm(g)(X,Y,Y,X)-h_{R}(X,X)h_{R}(Y,Y)+h_{R}(X,Y)^{2},
        \ee
        for all $X,Y\in T\Sigma_{R}$. Since $\Rm(g)$ is bounded on $\M$ and $h_{R}\rightarrow0$, these imply $\Rm(g_{R})$ is also bounded on $\Sigma_{R}$ as $R\geq R_{0}$.
        
        Now, since metric $\varphi_{R-R_{0}}^{*}g_{R}$ uniformly converge to $g_{\infty}$ and $\Rm(g_{R})$ is uniformly bounded, there is a constant $i_{0}>0$ such that $\inj(x,\varphi_{R-R_{0}}^{*}g_{R})\geq i_{0}$ for all $x\in\Sigma_{R_{0}}$ and $R\geq R_{0}$. Then the $\mathcal{C}^{1,\alpha}$-compactness theorem ensures that the convergence is $\mathcal{C}^{1,\alpha}$ for all $\alpha\in(0,1)$. Also, combining with $\Ric(g_{R})\rightarrow0$ as $R\rightarrow\infty$, using the harmonic coordinate w.r.t. $\varphi_{R-R_{0}}^{*}g_{R}$, which has uniform size and uniform $\mathcal{C}^{1,\alpha}$-bounds due to $\mathcal{C}^{1,\alpha}$-convergence, and the $\mathcal{C}^{1,\alpha}$-convergence, the limiting metric is smooth and Ricci flat, see the proof of \cite[Main Lemma 2.2]{A90} for a detailed argument.
        
        It remains to show the $\mathcal{C}^{\infty}_{\text{loc}}$-convergence on $\M$. 
        Consider the map $\Psi(x):=(F(x),\varphi_{F(x)-R_{0}}(x))$ from $\{z\geq R_{0}\}$ into $[0,\infty)\times\Sigma_{R_{0}}$. Then 
        \be
        \lim_{z\rightarrow\infty}\|(\Psi^{-1})^{*}g-(dz^{2}+g_{\infty})\|_{dz^2 +g_{\infty}}=0.
        \ee
        So $(\M,g,p_{k})$ converges to $(\R\times \Sigma_{R_{0}},dz^{2}+g_{\infty},(0,o))$ in the $\mathcal{C}^{0}$-sense if $p_{k}\rightarrow\infty$ and if $\varphi_{R_{0}-F(p_{k})}(p_{k})\rightarrow o$. Since $\Rm(g)$ is uniformly bounded and the $\mathcal{C}^{0}$-convergence implies the uniform lower bound of unit volume, these is a constant $i_{1}>0$ such that $\inj(x,(\Psi^{-1})^{*}g)\geq i_{1}$ for all $z(x)\geq R_{0}+i_{1}$. Hence, by \cite[Lemma 2.6]{D17} (or standard Shi's estimate), for any $k\geq1$, there is a constant $C_{k}$ such that $|{\nabla^{k}\Rm(g)}|_{g}\leq C_{k}$. By pointed Hamilton compactness theorem, there is a subsequence $(\M,g,p_{k})$ that converges to some pointed complete smooth manifold $(\N^{n},h,q)$ smoothly and locally. Since $(\M,g,p_{k})$ converges to $(\R\times\Sigma_{R_{0}},dz^{2}+g_{\infty},(0,o))$ in the $\mathcal{C}^{0}$-sense, this implies $(\N^{n},h,q)=(\R\times\Sigma_{R_{0}},dz^{2}+g_{\infty},(0,o))$. Therefore, the limit is unique and is $(\R\times\Sigma_{R_{0}},dz^{2}+g_{\infty},(0,o))$. If we further assume $\lim_{r\rightarrow\infty}|{\Rm}(g)|=0$, then by $(\ref{eq-Gauss eq of Rm})$, $|{\Rm(g_{R})}|_{g_{R}}\rightarrow0$ as $R\rightarrow\infty$. Then the smooth convergence deduces $(\R\times\Sigma_{R_{0}},dz^{2}+g_{\infty})$ is a flat manifold. 
\end{proof}
\begin{rmk}
    One can generalize to the case $|{\Ric}|=\BigO(h(r)/r)$, where $h:\R^{+}\rightarrow\R^{+}$ is a non-increasing function and $h(r)/r\in\mathcal{L}^{1}(\R^{+})$. Follow the same strategy as above, the limiting metric $g_{\infty}$ exists on level set $\Sigma_{R_{0}}$, $|{\Rm}|=\BigO(r^{\epsilon})$, and $|{\nabla\Ric}|=\BigO(r^{-1+\varepsilon/2})$ by \cite[Lemma 2.7]{D17}. Therefore, by Gauss equation again, the Ricci curvature of $\varphi_{R-R_{0}}^{*}g|_{\Sigma_{R}}$ tends to zero as $R\rightarrow\infty$. Adopting the same argument as above can show $\mathcal{C}^{1,\alpha}_{\text{loc}}$-convergence on the level set and asymptotic limit $p_{k}\rightarrow\infty$. However, without uniform Riemann curvature bounds $|{\Rm}|$, It is unclear to us how to derive the smooth convergence.
\end{rmk}
\begin{rmk}
    In Corollary \ref{cl-asymptotic limit of super-linear decay}, using the Gauss equation with a higher derivative, since $|{\nabla^{k}\Rm}(g)|_{g}\leq C_{k}$ on $\M$ and $h_{R}=\Hess(F)/|{\nabla F}|$, we can deduced uniform bounds on $\hat{\nabla}^{k}\Rm(\varphi_{R-R_{0}}^{*}g|_{\Sigma_{R}})$ on $\Sigma_{R_{0}}$, where $\hat{\nabla}$ is the Levi-Civita connection of $\varphi_{R-R_{0}}^{*}g|_{\Sigma_{R}}$. Therefore, combining with the $\mathcal{C}^{0}$-convergence, we could still conclude the smooth convergence on the level set.
\end{rmk}

\end{document}